\documentclass[a4paper,openright,twoside]{article}

\usepackage{amsfonts}
\usepackage{amsmath,amssymb}
\usepackage{amscd}
\usepackage{amsthm}
\usepackage{fancyhdr}
\usepackage{color}
\usepackage{hyperref}
\usepackage{tikz-cd}
\usetikzlibrary{decorations.pathmorphing}
\usetikzlibrary{arrows.meta}
\usepackage{adjustbox}
\usepackage{centernot}
\usepackage{mathtools}
\usepackage{stmaryrd}

\usepackage{comment}

\usepackage{tensor} 
\usepackage{cancel} 

\theoremstyle{plain}
\newtheorem{lemma}{Lemma}[section]
\newtheorem{theorem}[lemma]{Theorem}
\newtheorem{corollary}[lemma]{Corollary}
\newtheorem{proposition}[lemma]{Proposition}

\theoremstyle{definition}

\newtheorem{defnx}[lemma]{Definition}
\newenvironment{definition}
  {\pushQED{\qed}\defnx}
  {\popQED\enddefnx}

\newtheorem{examplex}[lemma]{Example}
\newenvironment{example}
  {\pushQED{\qed}\examplex}
  {\popQED\endexamplex}

  \theoremstyle{remark}
\numberwithin{equation}{lemma}

\newtheorem{obsx}[lemma]{Remark}
\newenvironment{remark}
  {\pushQED{\qed}\obsx}
  {\popQED\endobsx}

\newcommand{\NN}{\mathbb{N}}
\newcommand{\ZZ}{\mathbb{Z}}

\newcommand{\RR}{\mathbb{R}} 

\newcommand{\PP}{\mathbb{P}}
\newcommand{\TT}{\mathbb{T}}
\newcommand{\SSS}{\mathbb{S}}

\newcommand{\cG}{\mathcal{G}}
\newcommand{\cH}{\mathcal{H}}

\newcommand{\cO}{\mathcal{O}}

\newcommand{\g}{\mathfrak{g}}

\newcommand{\gl}{\mathfrak{gl}}
\renewcommand{\o}{\mathfrak{o}}
\renewcommand{\k}{\mathfrak{k}}

\newcommand{\tto}{\rightrightarrows}
\newcommand{\Ima}{\mathrm{Im}}
\newcommand{\Fr}{\mathrm{Fr}}
\newcommand{\GL}{\mathrm{GL}}
\newcommand{\pr}{\mathrm{pr}}
\newcommand{\rank}{\mathrm{rank}}
\newcommand{\id}{\mathrm{id}}

\newcommand{\Lie}{\mathrm{Lie}}
\newcommand{\Sp}{\mathrm{Sp}}
\newcommand{\SO}{\mathrm{SO}}

\newcommand{\gr}{\mathrm{gr}}
\renewcommand{\Gauge}{\mathrm{Gauge}}
\newcommand{\taut}{\mathrm{taut}}
\newcommand{\Gr}{\mathrm{Gr}}
\newcommand{\conn}{\mathrm{conn}}

\usepackage[style=numeric-comp, doi=false,isbn=false,url=false,eprint=false, maxnames=4]{biblatex} 
\AtEveryBibitem{%
  \clearlist{language}%
}

\begin{document}

\title{Sub-Riemannian structures and non-transitive Cartan geometries via Lie groupoids}

\author{Ivan Beschastnyi\footnote{Centre Inria d'Université Côte d'Azur, France. Email: \texttt{ivan.beschastnyi@inria.fr}} \and Francesco Cattafi\footnote{Julius-Maximilians-Universit\"at W\"urzburg, Germany. Email: \texttt{francesco.cattafi@uni-wuerzburg.de}} \and João Nuno Mestre\footnote{CMUC, University of Coimbra, Department of Mathematics, Portugal. Email: \texttt{jnmestre@proton.me}}}

\maketitle

\abstract{In this paper we discuss how to associate a suitable non-transitive version of a Cartan connection to sub-Riemannian manifolds of corank 1 (including contact and quasi-contact sub-Riemannian manifolds) with non-necessarily constant sub-Riemannian symbols.

In particular, we recast the variation of the sub-Riemannian symbols into a suitable "type" map, which is constant if and only if the symbols are constant. We then consider the (non-transitive) groupoid of sub-Riemannian symmetries and investigate its smoothness, properness, regularity, and other properties in relation with the type map. Last, we describe how to build a "non-transitive" analogue of a Cartan connection on top of such (Lie) groupoid, obtained as the sum of a tautological form with a multiplicative Ehresmann connection. We conclude by illustrating our results on concrete examples in dimension 5.
}

 \begin{center}
 \end{center}

\tableofcontents

\section{Introduction}

This is the first paper in a series devoted to the study of sub-Riemannian manifolds and their local invariants, employing as key tools the theory of Lie groupoids and Lie algebroids. This first paper focuses on understanding how to associate a (suitable generalisation of a) Cartan geometry to a given sub-Riemannian manifold --- a long-standing problem (cf. \cite[Remark 4]{Morimoto08}) which until now was solved only in the case of structures with constant sub-Riemannian symbols.

\paragraph{Cartan geometries and local invariants}

Cartan geometries originated in Élie Cartan's studies on \textit{espaces généralisés} and in Felix Klein's proposal that each "geometry" should be described by groups of transformations (the so-called Erlangen programme). More precisely, Klein geometries determine the "model space" for arbitrary Cartan geometries; in turn, Cartan geometries which are flat become locally isomorphic to their underlying Klein model.

In modern language, Klein geometries are simply homogeneous spaces, while Cartan geometries can be formulated as principal bundles together with special kinds of vector valued 1-forms, called \textbf{Cartan connections} \cite{Sharpe97, CapSlovak09}. Cartan geometries constitute a general framework for studying various classes of geometric structures with the same formalism. Furthermore, they constitute an important tool to extract local invariants from geometric structures on smooth manifolds. The principle behind that is the following: if a geometric structure yields a (possibly canonical) Cartan geometry, one can use the \textbf{curvature} of such Cartan connection to obtain invariants of the original structures. The general scheme is therefore
\[\begin{tikzcd}
	{\text{geometric structure on } M} & {\text{Cartan geometry on } M} & {\text{local invariants on } M.}
	\arrow["{{?}}", squiggly, from=1-1, to=1-2]
	\arrow["{{?}}", squiggly, from=1-2, to=1-3]
\end{tikzcd}\]
Let us illustrate this process in a simple example (which of course can be formulated more directly, without using this formalism). Any Riemannian manifold admits a canonical Cartan connection (built using the Levi-Civita connection), and its only invariant is precisely its curvature (which coincides with the curvature of the Levi-Civita connection). More precisely, given a Riemannian metric $g$, one can build a Cartan geometry $(P,\theta)$ modelled on the Euclidean geometry $(O(n), \RR^n \rtimes O(n))$ as follows:
\begin{itemize}
 \item $P \subseteq \Fr(M)$ is the bundle of orthonormal frames of $g$, which is naturally an $O(n)$-principal bundle;
 \item $\theta \in \Omega^1 (P, \RR^n \oplus \o(n))$ is the sum between the tautological form $\theta_{\taut} \in \Omega^1 (P, \RR^n)$ and the Levi-Civita connection $\theta_{\conn} \in \Omega^1 (P,\o(n))$.
\end{itemize}
The curvature $d\theta + 1/2 [\theta,\theta] \in \Omega^1 (P, \RR^n \oplus \o(n))$ of the Cartan connection $\theta$ is identified with the curvature $d\theta_{\conn} + 1/2 [\theta_{\conn}, \theta_{\conn}] \in \Omega^2 (P, \o(n))$ of the Levi-Civita connection (since its $\RR^n$-valued component coincides with the torsion of $\theta_{\conn}$, which is zero). In turn, this is precisely the curvature of $g$, which is the only local invariant of $g$.

More involved examples of geometric structures which benefit from this point of view involve conformal structures, projective structures, and, more general, parabolic geometries \cite{CapSlovak09, Cap17}.

\paragraph{Cartan geometries and Cartan groupoids}

It is well known that transitive Lie groupoids are, up to isomorphism, completely encoded by principal bundles \cite{MoerdijkMrcun03, Mackenzie05}. Indeed, any principal $G$-bundle $P \to M$ induces the \textbf{gauge groupoid} $\Gauge(P) := (P \times P)/G \tto M$. Conversely, given a transitive Lie groupoid $\cG \tto M$, the $s$-fibre $s^{-1}(x)$ over any point $x \in M$ is a principal $\cG_x$-bundle, and choosing a different $y \in M$ yields an isomorphic bundle with structure group $\cG_y \cong \cG_x$. An important instance of this correspondence is given by the frame bundle $P = \Fr(M) \to M$ and the general linear groupoid $\cG = \GL(TM) \tto M$.

This classical correspondence between principal bundles and transitive Lie groupoids can be upgraded by adding further structures on both sides. In particular, one can endow $P$ with a Cartan connection and look at what is obtained on the other side on $\cG = \Gauge(P)$. The result is a 1-form $\omega \in \Omega^1 (\cG,t^*A)$ with coefficients in the Lie algebroid $A = \Lie(\cG)$, which is multiplicative and satisfies $\ker(\omega) \cap \ker(s) = \ker(\omega) \cap \ker(t) = 0$. The pair $(\cG,\omega)$ is called a \textbf{Cartan groupoid} \cite{Blaom16, Cattafi21}; these objects can also be seen as a special case of Pfaffian groupoids \cite{Salazar13, Bookpseudogroups}.

The scheme presented above can be therefore reformulated as
{\footnotesize
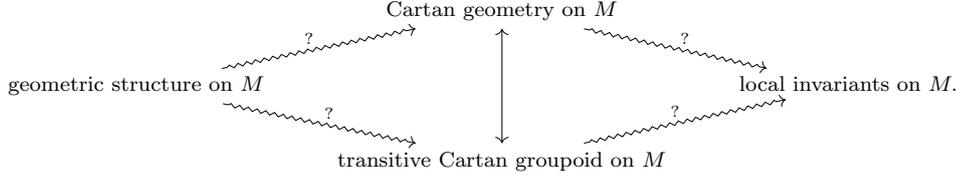
\[\begin{tikzcd}
	& {\text{Cartan geometry on } M} \\
	{\text{geometric structure on } M} && {\text{local invariants on } M.} \\
	& {\text{transitive Cartan groupoid on } M}
	\arrow["{?}", squiggly, from=1-2, to=2-3]
	\arrow[<->, from=1-2, to=3-2]
	\arrow["{?}", squiggly, from=2-1, to=1-2]
	\arrow["{?}", squiggly, from=2-1, to=3-2]
	\arrow["{?}", squiggly, from=3-2, to=2-3]
\end{tikzcd}\]
}

On the other hand, there are many geometric structures which exhibit ``non-transitive'' behaviours, i.e.\ their local model may change from point to point. In other words, a single principal bundle and a single Cartan connection are not enough to encode them. Nevertheless, (non-necessarily transitive) Lie groupoids offer precisely the right techniques to tackle this issue. Accordingly, the main scheme becomes now

{\footnotesize
\[\begin{tikzcd}
	{\text{geometric structure on } M} & {\text{(non-transitive) Cartan groupoid on } M} & {\text{local invariants on } M.}
 	\arrow["{?}", squiggly, from=1-1, to=1-2]
 	\arrow["{?}", squiggly, from=1-2, to=1-3]
\end{tikzcd}\]
}

\paragraph{Cartan geometries and sub-Riemannian structures}

In this paper we initiate the general investigation of the point of view presented above, focusing on the first question mark (the construction of Cartan groupoids) and on the case of sub-Riemannian structures of corank one.

Recall that a \textbf{sub-Riemannian manifold} $(M, D, g)$ consists of a manifold $M$ equipped with a bracket-generating distribution $D \subseteq T M$ and a Riemannian metric $g$ defined along $D$ (but not necessarily along $TM$) \cite{Montgomery02, AgrachevBarilariBoscain20}. Distinguished examples include the case when the distribution $D$ is a \textbf{contact structure} (on odd-dimensional manifolds) or a quasi-contact/even-contact structure (on even-dimensional manifold) \cite{LibermannMarle87, Geiges08}.

The geometry of an arbitrary sub-Riemannian manifold is controlled by its \textbf{sub-Riemannian symbols}, i.e. the pairs
\[
(\gr(T_xM), g_x), \quad \quad x \in M,
\]
consisting of the nilpotent Lie algebras $\gr(T_xM): = T_xM/D_x \oplus D_x$ and the metric $g_x$ on $D_x$. Indeed, the simply connected Lie group integrating the Lie algebra $\gr(T_xM)$ should be thought of as model space for the sub-Riemannian manifold $(M,D,g)$ at a point $x \in M$. This can be made formal via the Gromov-Hausdorff limit of metric spaces~\cite{AgrachevBarilariBoscain20,Mitchell85,Montgomery02}.

Of particular importance is the case when $(M,D,g)$ have \textbf{constant symbol}, meaning that, for every $x,y \in M$, the Lie algebras $\gr(T_xM)$ and $\gr(T_y M)$ are isomorphic, and this isomorphism restricts to an isometry between $(D_x, g_x)$ and $(D_y, g_y)$. In this case, since there is a single "local model", one can use the theory of prolongations of filtered $G$-structures in order to build a canonical Cartan connection \cite{Morimoto93, Morimoto08, Cap17, Grong2020}. In turn, the curvature of such Cartan connection can be used to produce several invariants of the original structure on $M$ \cite{sub-lorentzian,Morimoto08}.

On the other hand, if the sub-Riemannian symbols $(\gr(T_xM), g_x)$ do depend on $x$, no canonical Cartan connection has been constructed so far. The reason is that principal bundles do not contain enough information to encode a geometric structure with a point-dependent model; this was noticed e.g.\ in \cite[Remark 4]{Morimoto08} and \cite[Section 8.1]{HongMorimoto24}. This issue suggests that a solution to the general questions formulated above should employ (not necessarily transitive) Lie groupoids.

\paragraph{Main results and structure of the paper}

In the first part of this paper (Section \ref{section_subriemannian_structures}) we present an overview on bracket-generating distributions, corank-one structures (including contact and quasi-contact structures), and sub-Riemannian geometry. While most of this material is standard, in Subsection \ref{section_type_map} we provide a thorough study of the \textbf{type map} of sub-Riemannian corank-one structures. This is a map \eqref{eq_type_D}
\[
\alpha: M \to \RR \PP^{n-1},
\]
where $\dim(M) = 2n+1$ or $2n+2$, defined by the eigenvalues of a bundle morphism $J: D \to D$ \eqref{eq_map_J} naturally associated with the metric $g$ and the Levi map $\mathcal{L}_D$ of the distribution $D$.

The notion of type seems to belong to the folklore of the sub-Riemannian community, but it lacks a systematic treatment. To the best of our knowledge, it has only been briefly touched upon in the literature
\cite{Savale17, Ludovic2019}. We therefore take this opportunity to provide a precise definition and a detailed investigation of the concept. In particular, we prove explicitly that the type contains the same information as the sub-Riemannian symbols (Theorem \ref{prop_orbits_without_groupoids}); it follows that sub-Riemannian manifolds have constant symbol if and only if their type map is constant. We also discuss how the image $\alpha(M) \subseteq \RR \PP^{n-1}$ interacts with the natural stratification of the $(n-1)$-simplex $\Delta^{n-1}$. We illustrate these concepts on a flexible class of 5-dimensional sub-Riemannian structures parametrised by a smooth function (Examples \ref{ex_Charlotte}, \ref{ex_charlotte_isotropy} and \ref{ex_subriemannian_not_constant_symbols}).

\

Next, in Section \ref{section_groupoids_for_subriemannian}, we investigate the interplay between sub-Riemannian manifolds $(M,D,g)$ and their groupoids of symmetries. We look first at arbitrary (bracket-generating) distributions, and then specialise to corank-one and contact sub-Riemannian structures. In particular, we prove that, under minimal assumptions on the type map $\alpha$, the groupoid 
\[
 \cG := \left\{ \gamma_{xy} \in \GL^{\gr} (D) |
 \begin{array}{l}\gamma_{xy}: \gr(T_xM) \to \gr(T_yM) \text{ is a Lie algebra isomorphism and }\\
\gamma_{xy}^{-1}: (D_x,g_x) \to (D_y,g_y) \text{ is an isometry }
\end{array} \right\}
\]
is a proper regular Lie groupoid (Theorem \ref{prop_G_regular_Lie_subgroupoid} and Proposition \ref{prop_G_proper}). We examine as well the subgroupoid $\cG' \subseteq \cG$ of elements preserving the foliation given by the fibres of $\alpha$: understanding its smoothness leads us to the definition of a secondary type map $\beta$ \eqref{eq_secondary_type_D}. Indeed, we prove that, under minimal assumptions on $\alpha$ and $\beta$, the groupoid $\cG'$ is always smooth and proper in dimension 5 (Theorem \ref{prop:classification_5d}).

\

Last, Section \ref{section_cartan_connections_subriemannian_contact_manifolds} deals with the procedures to associate a Cartan groupoid to a sub-Riemannian structure. In Subsection \ref{subsection_principal_connections} we review the classical notion of principal connection, as as well as its groupoid counterpart, called a \textbf{multiplicative Ehresmann connection} \cite{FernandesMarcut23}. Similarly, in Subsection \ref{section_Cartan_connections} we recall the notion of Cartan connection and of Cartan groupoid \cite{Blaom16, Cattafi21}. In particular, we explain that the sum of the \textbf{tautological form} $\omega_{\taut} \in \Omega^1 (\GL(TM), t^*TM)$ of $\GL(TM)$ and any multiplicative Ehresmann connection $\omega_{\conn} \in \Omega^1 (\GL(TM), t^* \ker(\rho))$ \cite{FernandesMarcut23} yields a Cartan groupoid, analogously to the fact that the sum of the tautological form of $\Fr(M)$ and any principal connection yields a Cartan geometry. The same argument works for a \textit{transitive} Lie subgroupoid $\cG \subseteq \GL(TM)$, while non-transitive ones involve additional subtleties concerning the coefficients of the relevant 1-forms.

In Subsection \ref{subsection_graded_Cartan_groupoids} we extend these notions and these strategies to the graded setting, in order to apply it to the groupoids $\cG$ and $\cG'$ associated to sub-Riemannian manifolds. Here the properness of $\cG$ and $\cG'$ is instrumental to ensure the existence of a multiplicative Ehresmann connection. On the other hand, the Lie groupoid $\GL^\gr(D)$, i.e.\ the graded version of $\GL(TM)$, is endowed as well with a tautological form, whose restriction to $\cG$ requires again extra care regarding the coefficients.

If $\cG$ coincides with $\cG'$ (the subgroupoid preserving $\ker(d\alpha)$), one does get immediately a (graded) Cartan groupoid $(\cG, \omega_{\taut} + \omega_{\conn})$ with the same strategy as in the non-graded case; this happens for instance if $\cG$ is transitive, i.e.\ if the structure has constant symbols, recovering therefore the classical results. However, in general $\cG$ might fail to carry a structure of Cartan groupoid; one should restrict the relevant form on $\cG'$ and repeat the process; in turn, this might induce a series of reductions of the groupoid $\cG$, which stops precisely when one obtains a (graded) Cartan groupoid. Our main concluding result (Theorem \ref{prop_final}) is that, if one starts with a 5-dimensional sub-Riemannian corank-one structure satisfying mild conditions on $\alpha$ and $\beta$ making $\cG'$ smooth, this procedure stops after at most one reduction.

\paragraph{Further motivations and future applications}

In follow-up works we intend to develop the appropriate notion of curvature for (graded) Cartan groupoids and its interplay with the type map, in order to facilitate the calculation of invariants of sub-Riemannian structures. In particular, we aim to do this making use of a Chern-Weil construction for Cartan groupoids \cite{AccorneroDeMeloStruchiner26}. We also plan to recover the results obtained in dimensions 3 and 4 with more ad hoc techniques  \cite{contact3d,contact4d}.

On the other hand, our findings will provide an interesting theoretical framework for further investigations of other geometric structures which exhibits "non-transitive phenomena", starting with sub-Riemannian structures of arbitrary rank.

Last, through the paper we briefly touched several topics that deserved to be studied separately in the future: these include the existance of sub-Riemannian structures with given image of the type map (Remark \ref{rmk_realization_problem}),  relations with gerbes (Remark \ref{rk_gerbes}), a general prolongation and reduction algorithm for graded Cartan groupoids, and the general theory of filtered and graded Lie groupoids and algebroids.

\paragraph{Notations and conventions}

All manifolds and maps are smooth, unless explicitly stated otherwise. Lie groupoids are always assumed to be Hausdorff and 2nd countable.

When describing elements $\gamma_{xy}: \gr(T_xM) \to \gr(T_yM)$ of the graded general linear groupoid $\GL^{\gr}(TM) \tto M$, i.e.\ isomorphisms of graded vector spaces, the notation $\gamma_{xy}^{-1}$ denotes the $(-1)$-component $D_x \to D_y$ of $\gamma_{xy}$, and not the inverse of $\gamma_{xy}$.

\paragraph{Acknowledgements}

This article is based upon work from COST Action CaLISTA CA21109 supported by COST (European Cooperation in Science and Technology). \href{www.cost.eu}{www.cost.eu}. In particular, the authors acknowledge financial support for two research stays in Coimbra (March 2025) and Nice (January 2026).

The authors thank Ivan Struchiner and Luca Accornero for useful comments and interesting discussions at various stages of this project.

Francesco Cattafi would like to thank Andreas \v{C}ap for pointing out to him the existing links between Cartan geometries and sub-Riemannian contact structures, and for other useful discussions and clarifications during his Junior Research Fellowship at the Erwin Schrödinger Institute (Austria). Furthermore, he was partially supported by the DFG under Walter Benjamin project 460397678 (Germany), and is a member of the GNSAGA - INdAM (Italy).

Ivan Beschastnyi was partially supported by by the French government through the France 2030 investment plan managed by the National Research Agency (ANR), as part of the Initiative of Excellence Université Côte d’Azur under reference number ANR-15-IDEX-01.

João Nuno Mestre was partially supported by the Centre for Mathematics of the University of Coimbra (CMUC, https://doi.org/10.54499/UID/00324/2025) under the Portuguese Foundation for Science and Technology (FCT), Grants UID/00324/2025 and UID/PRR/00324/2025.

\section{Sub-Riemannian structures}\label{section_subriemannian_structures}

\subsection{Distributions}

We recall here the basics on (regular) distributions: more details and examples can be found e.g.\ in \cites{delPino19, MontgomeryZhitomirskii01}.

\begin{definition}\label{def_equiregular_equinilpotent_distributions}
 Let $D \subseteq TM$ be a regular distribution, i.e.\ a (smooth) subbundle. We consider the filtration of subspaces
\[
 D^{-1} \subseteq D^{-2} \subseteq \ldots \subseteq D^i \subseteq D^{i-1} \subseteq \ldots
\]
defined recursively at $x\in M$ by
 \[
D_x^{-1} := D_x, \quad \quad D_x^i := D_x^{i+1} + [D, D^{i+1}]_x \subseteq T_xM,
 \]
where $[D, D^{i+1}]$ denotes $[\Gamma(D), \Gamma(D^{i+1})]\subset \Gamma(TM)$. Note that the distributions $D^i \subseteq TM$ are not necessarily regular, since their rank may vary.

For any $x \in M$, consider the vector space
    \[
   \gr(T_xM, D_x) = \gr(T_xM) := \bigoplus_{k={-\infty}}^{-1} D^k_x/D^{k+1}_x.
  \]
  where we set $D^0_x := \{0\}$. This space has a natural structure of a nilpotent Lie algebra defined as follows. Take $v\in D^i_x/D^{i+1}_x$ and $w\in D^j_x/D^{j+1}_x$ and extend them to two vector fields $X_v \in \Gamma(D^i)$ and $X_w \in \Gamma(D^j)$. Then the Lie bracket on $\gr(T_xM)$ is defined as
  $$
  [v,w] = [X_v,X_w]_x \mod D_x^{i+j+1}.
  $$
  One can check that the bracket is well-defined, i.e. does not depend on the extensions $X_v,X_w$ of $v,w$ and that one gets a structure of a nilpotent Lie algebra. The space $\gr(T_xM)$ is called the \textbf{graded tangent space at $x$} or the \textbf{osculating Lie algebra at $x$}. When we want to stress the dependence on $D$ (e.g.\ in presence of several distributions) we will use the notation $\gr(T_x M, D_x)$.

 \begin{itemize}
  \item A distribution $D$ is called \textbf{bracket-generating} if there is some $r \in \ZZ$ such that $D^{r} = T M$ i.e.\ the sequence of subspaces $\{D^i_x\}_{i=-\infty}^{-1}$ stabilises to $T_x M$ after a finite number $r(x)\leq r$ of steps which may depend on $x$. Concretely, for each $x\in M$, let $r(x) \in \ZZ$ be such that $D^{r(x)+1}_x \neq T_x M$, but $D^{r(x)}_x = T_x M$ (assuming that $\mathrm{dim}(M)>0$).
  The number $r = \max_{x \in M} |r(x)| \in \NN$ is called the \textbf{step} of $D$ (also \textbf{depth}, \textbf{size} in other literature).
  \item A distribution $D$ is called \textbf{equiregular} (or \textbf{weakly regular}) if $D^i$ is a regular distribution for every $i$.
  \item A distribution $D$ is called \textbf{equinilpotent} (or \textbf{strongly regular}) if for all $x \in M$ the associated nilpotent Lie algebras $\gr(T_xM)$ are isomorphic (as Lie algebras) to each other. Note that equinilpotent distributions are automatically equiregular. \qedhere
 \end{itemize}
\end{definition}

\begin{example}
 \textbf{Contact distributions} $D \subseteq TM$ (reviewed in Definition \ref{def_corank1_contact_quasicontact} below) satisfy by definition $D^{-2} = TM$: they are therefore both bracket-generating of step 2 and equiregular. Furthermore, they are equinilpotent because all the Lie algebras $\gr(T_xM) = D_x^{-2}/D_x^{-1}\oplus D_x^{-1} = T_xM/D_x\oplus D_x$ are isomorphic to the Heisenberg Lie algebra.
\end{example}

\begin{example}
 \textbf{Involutive distributions} $D \subseteq TM$ satisfy by definition $D^i = D$ for every $i$; they are therefore trivially equiregular, but not bracket-generating. They are also trivially equinilpotent because $\gr(T_xM)$ is the space $D_x$ with the trivial bracket for every $x \in M$.
\end{example} 

\begin{example}
    Consider $M = \RR^3$ with coordinates $(x,y,z)$ and the \textbf{Martinet distribution} \cite{Martinet70}
 \[
  D := \ker(dy - z^2 dx) = \langle \partial_x + z^2 \partial_y, \partial_z \rangle \subseteq TM.
 \]
 Then $D^{-2}$ has rank 3 everywhere except on the hypersurface $\{z = 0\}$ (where the rank is 2), while $D^{-3} = TM$. It follows that $D$ is bracket-generating (of step 3) but not equiregular (hence not equinilpotent).
\end{example}

\begin{example}
\label{ex_Charlotte}
    A simple example of an equiregular but not equinilpotent distribution is given by a distribution on $M = \mathbb R^5$ with coordinates $(x_1,x_2,y_1,y_2,z)$ defined as follows (see \cite[Example 2.33]{delPino19}):
\begin{equation}
\label{eq:example_charlotte_generators}
    D: = \ker(dz - y_1d x_1 - f(y_2)dx_2) = \langle \partial_{y_1},\partial_{y_2},\partial_{x_1}+y_1\partial_z,\partial_{x_2}+f(y_2)\partial_z\rangle, 
\end{equation}
where $f: \RR \to \RR$ is any smooth function. To find the corresponding Lie algebra, we compute the Lie brackets between all the basis vector fields listed above. The only non-zero brackets are given by
\begin{align*}
    &[\partial_{y_1},\partial_{x_1}+y_1\partial_z] = \partial_z\\
    &[\partial_{y_2},\partial_{x_2}+f(y_2)\partial_z] = f'(y_2)\partial_z.
\end{align*}
The distribution $D$ is bracket-generating of step 2. We can identify $TM/D = \langle \partial_z\rangle$. Then at the points where $f'(y_2)\neq 0$ the nilpotent Lie algebra has a one-dimensional center (given by the direction $\partial_z$), while at the points where $f'(y_2)=0$ the center is three-dimensional (given by the span of the directions $\partial_z$, $\partial_{y_2}$ and $\partial_{x_2}+f(y_2)\partial_z$). Therefore, the two nilpotent Lie algebras are not isomorphic even though the distribution is equiregular. More precisely, at the points where $f'(y_2)\neq 0$ the osculating Lie algebra is isomorphic to the 5D Heisenberg Lie algebra, while at the points where $f'(y_2)= 0$ it is isomorphic to the product of a 3D Heisenberg algebra and 2D abelian Lie algebra.
\end{example}

\subsection{Corank-one, contact and quasi-contact structures}

We recall now the basics on contact geometry and its variations; for more details, we refer to the classical textbooks \cite{LibermannMarle87, Geiges08}.

\begin{definition}\label{def_corank1_contact_quasicontact}
Consider a manifold $M$ and a distribution $D \subseteq TM$ of corank 1, i.e.\ $\rank D = \dim M - 1$. Its \textbf{Levi map} (or Levi form, or curvature) is the bilinear skew-symmetric map
     \[
 \mathcal{L}_D: D \times D \to TM/D, \quad \quad (X,Y) \mapsto [X,Y] \mod D.
 \] 
 The distribution $D$ is called
 \begin{itemize}
     \item a \textbf{corank-one structure} if $\mathcal{L}_D$ not zero; equivalently, $D$ is not involutive;
     \item a \textbf{contact structure} if the kernel
     \[
     \ker(\mathcal{L}_D) := \{ X \in D | \mathcal{L}_D (X,Y) = 0 \ \forall Y \in D \}
     \]
     is zero-dimensional; equivalently, $\mathcal{L}_D$ is non-degenerate;
     \item a \textbf{quasi-contact structure} (or even-contact structure) if $\ker(\mathcal{L}_D)$ is one-dimensional. \qedhere
 \end{itemize}
\end{definition}
 Note that corank-one structures are automatically bracket-generating (of step 2) and equiregular (Definition \ref{def_equiregular_equinilpotent_distributions}), while contact and quasi-contact structures are automatically corank-one.

\begin{remark}[dimensional constraints]
Contact and quasi-contact structures force the underlying manifold to have, respectively, odd and even dimension. Indeed, the Levi map $\mathcal{L}_D$ is always skew-symmetric; if, in addition, it is non-degenerate, then $D$ must have an even rank, and therefore $\dim M = 2n+1$ for $n\geq 1$. Similarly, if $\mathcal{L}_D$ has a one-dimensional kernel, then $D$ must have an odd rank, and therefore $\dim M = 2n+2$ for $n\geq 1$. One can say that contact and quasi-contact structures are distinguished from general corank-one structures by having the smallest possible kernel of the Levi map.
\end{remark}

Corank-one structures can be described, at least locally, as the kernel of a one-form $\theta \in \Omega^1(M)$ such that 
\begin{equation}
    \label{eq:nonintegrability_general_case}
\theta\wedge d\theta \neq 0.
\end{equation}
 If $M$ has dimension $2n+1$ or $2n+2$ and the corank-one structure is, respectively, contact or quasi-contact, then the form $\theta$ satisfies
\begin{equation}\label{eq_nonintegrability_contact_quasicontact}
 \theta\wedge (d\theta)^n \neq 0, 
\end{equation}
 and is called, respectively, a \textbf{contact} or a \textbf{quasi-contact form}.

\begin{remark}[coorientable structures]\label{rk_coorientable_structures}
Any one-form $\theta$ satisfying~\eqref{eq:nonintegrability_general_case} defines a corank-one structure $\ker(\theta) \subseteq TM$; conversely, a corank one structure is called {\bf coorientable} if it can be written in this way, for some (not necessarily unique) form $\theta$ satisfying~\eqref{eq:nonintegrability_general_case}. Note that the one-form $\theta$ defines a distribution only up to a non-vanishing function $f\in C^\infty(M)$, meaning that $f\theta$ and $\theta$ define the same distribution $D$. The same holds for contact and quasi-contact structures if we replace~\eqref{eq:nonintegrability_general_case} with~\eqref{eq_nonintegrability_contact_quasicontact}.
\end{remark}

\begin{remark}[contact vs symplectic structures]\label{rk_symplectic_vector_bundle}
Let $\theta$ be a contact form on $M^{2n+1}$. Then the associated contact distribution $\ker(\theta)$, together with the section
\begin{equation}\label{eq_omega_d_thera}
 \omega := d\theta|_{\ker(\theta)} \in \Gamma (\wedge^2 (\ker(\theta)^*)),
\end{equation}
is a \textbf{symplectic vector bundle} over $M$, i.e.\ $(\ker(\theta)_x, (d\theta|_{\ker(\theta)})_x)$ is a symplectic vector space for every $x \in M$. On the other hand, if $\ker(\theta)$ was an arbitrary corank-one structure, the bilinear forms $(d\theta|_{\ker(\theta)})_x$ would be skew-symmetric, but not necessarily non-degenerate. In fact, $\theta$ is contact if and only if $(\ker(\theta),\omega)$ is a symplectic vector bundle.

In the non-coorientable case one needs to replace $d\theta|_D$ with the Levi map (Definition \ref{def_corank1_contact_quasicontact}), which can be seen as a "2-form" $\mathcal{L}_D \in \Gamma (\wedge^2 D^*, TM/D)$ with values in the (non-trivial) line bundle $TM/D \to M$.
By choosing a local trivialisation of $TM/D$ one recovers precisely \eqref{eq_omega_d_thera}. On the other hand, from a global point of view, everything works analogously, but one needs to carry the cumbersome coefficient $TM/D$. For this reason, we will often restrict for simplicity to the coorientable case.
\end{remark}

\subsection{Corank-one sub-Riemannian structures and their type map}\label{section_type_map}

There is an extensive literature on sub-Riemannian structures: see e.g.\ \cite{Montgomery02, AgrachevBarilariBoscain20} and the references therein. However, the notion of "type" \eqref{eq_type_D} discussed below seems to be well-known in the community as folklore knowledge, but never explicitly defined. The only instances we could find where an analogous notion is (briefly) treated are \cite{Savale17, Ludovic2019}. We take therefore this occasion to define and investigate this concept in detail.

\begin{definition}\label{def_subriemannian_manifold}
 A {\bf sub-Riemannian manifold} is a triple $(M,D,g)$ consisting of a manifold $M$, a bracket-generating distribution $D$ and a Riemannian metric $g$ on $D$, i.e., a section $g\in\Gamma( S^2 D^*)$ which is non-degenerate and positive definite. The pair $(D,g)$ is called a \textbf{sub-Riemannian structure} on $M$.

 If $D$ is corank-one, contact, or quasicontact (Definition \ref{def_corank1_contact_quasicontact}), we call $(D,g)$, respectively, a \textbf{corank-one, contact, or quasicontact sub-Riemannian} structure.
\end{definition}

Consider a coorientable corank-one sub-Riemannian manifold $(M, D = \ker(\theta), g)$; as mentioned at the end of Remark \ref{rk_symplectic_vector_bundle}, the entire discussion below can be reformulated replacing $d\theta|_D$ with $\mathcal{L}_D$, but we will not pursue this generality for ease of notation.

Thanks to the non-degeneracy of $g$, there is a bundle morphism
\begin{equation}\label{eq_map_J}
 J: D \to D
\end{equation}
uniquely defined by the condition
\[
\omega(X,Y):= (d\theta)|_D (X,Y) = g (JX, Y) \quad \quad \forall X, Y \in \Gamma(D).
\]
Note that, even if $\omega$ is non-degenerate (i.e.\ $D$ is a contact distribution, as discussed in Remark \ref{rk_symplectic_vector_bundle}), it is not necessarily compatible with $g$ (in the sense of K\"ahler geometry), so $J$ is not necessarily an almost complex structure, i.e.\ $J^2 \neq - \id_D$. Notice also that $J$ is invertible if and only if $\omega$ is nondegenerate, i.e.\ $D$ is contact.

Furthermore, $J_x$ is skew-symmetric with respect to $g_x$, i.e.\  $g_x (J_x v, w) = - g_x(v, J_x w)$. It follows that the eigenvalues of $J_x$ have to be purely imaginary. Assume for simplicity that $\dim(M) = 2n+1$, i.e. $D$ has rank $2n$; then the eigenvalues have to be of the type $\{ \pm i \alpha_j^x \}_{j=1}^n$, for $\alpha_j^x \in \RR$; without loss of generality, we can assume $0\leq \alpha_1^x\leq\alpha_2^x\leq \cdots \leq\alpha_n^x$. Furthermore, at least one of the eigenvalues must be non-zero; otherwise, the distribution $D$ would not be bracket-generating of step 2 at $x$. Then one considers the map
\begin{equation}\label{eq_type_D}
\alpha: M^{2n+1} \to \RR \PP^{n-1}, \quad x \mapsto [\alpha_1^x : \ldots : \alpha_n^x],
\end{equation}
which in general may not be smooth. If $\dim(M) = 2n+2$, i.e.\ $D$ has rank $2n+1$, the same arguments hold, and the extra eigenvalue of $J_x$ is automatically zero, since a skew-symmetric form in an odd-dimensional space always has a kernel.

\begin{definition}
 The map $\alpha$ above \eqref{eq_type_D} is called the \textbf{type} of a corank-one sub-Riemannian structure $(\ker(\theta), g)$.
\end{definition}

\begin{remark}[dependence on the one-form]\label{rmk:uniqueness_of_J}
The definition of the type does not depend on the specific one-form $\theta$, but only on the entire distribution $\ker(\theta)$ (and the metric $g$). To see this, recall from Remark \ref{rk_coorientable_structures} that, if a corank-one distribution is generated by a form $\theta$, any other form which generates it is of the kind $\theta' = f\theta$, for $f \in \mathcal{C}^\infty(M)$ a non-zero function, i.e.\ $\ker(\theta) = \ker(f\theta)$. Then the skew-symmetric map \eqref{eq_map_J} induced by $\theta'$ and $g$ becomes
 \[
  J' = \frac{1}{f} J,
 \]
 so that its eigenvalues satisfy
 \[
  (\alpha')^x_i = \frac{1}{f (x)} \alpha_i^x.
 \]
 As a consequence, rescaling the projective coordinates, one gets $\alpha = \alpha'$.
\end{remark}

\begin{remark}[image of the type]\label{rk_image_type}
Consider the standard $(n-1)$-simplex
\[
 \Delta^{n-1} := \{ (0\leq \lambda_1\leq \dots \leq \lambda_n) \in \RR^n | \sum_{i=1}^n\lambda_i = 1 \} \subseteq \RR^n.
\]
The type $\alpha: M \to \RR \PP^{n-1}$ takes values in the image of the embedding
\begin{equation}\label{eq_embedding_simplex}
 \Delta^{n-1} \to \RR \PP^{n-1}, \quad (\lambda_1, \ldots, \lambda_n) \mapsto [\lambda_1 : \ldots : \lambda_n].
\end{equation}
We will then always identify $\Delta^{n-1}$ with its image by this embedding.
\end{remark}

We provide now a motivation for the type map, which will be fundamental when studying the corresponding Lie groupoids in Section \ref{section_groupoids_for_subriemannian}. Assume that we would like to understand when there exists a sub-Riemannian isometry $\varphi: (M,D,g) \to (M',D',g')$ between two corank-one sub-Riemannian structures, i.e.\ a diffeomorphism $\varphi: M\to M'$ satisfying $\varphi_* D = D'$ and $\varphi^*g' = g$. Given any local diffeomorphism $\varphi: M \to M'$ such that $\varphi(x) = y$, we obtain an induced map $\gamma_{xy}: \gr(T_{x}M,D_x)\to \gr(T_{y}M',D'_y)$. If the map $\varphi$ was a local isometry from a neighbourhood of $x$ to a neighbourhood of $y$, then the restriction $\gamma_{xy}: (D_{x},g_{x})\to (D'_{y},g'_{y})$ would be an isometry. Therefore, the first step is to understand when $\gamma_{xy}: \gr(T_x M, D_{x})\to \gr(T_y M', D'_{y})$ is a metric isomorphism of Lie algebras. The answer is given by the following proposition.

\begin{theorem}\label{prop_orbits_without_groupoids}
Let $(M,D,g)$ and $(M',D',g')$ be corank-one sub-Riemannian manifolds. Given any $x, y \in M$, the following properties are equivalent:
\begin{itemize}
\item $\alpha(x) = \alpha(y)$;
    \item there is a Lie algebra isomorphism $\gamma_{xy}:\gr(T_x M, D_{x})\to \gr(T_y M', D'_{y})$ such that the component $\gamma_{xy}^{-1}: (D_x,g_x) \to (D'_y,g'_y)$ is an isometry;
\item there is an isometry $\gamma_{xy}^{-1}: (D_x,g_x) \to (D'_y,g'_y)$ which additionally preserves the Levi form, in the sense that there is there is a $\gamma_{xy}^{-2}: T_xM/D_x \to T_yM/D_y$ such that for all $X,Y\in D_x$
\begin{equation*}\label{preserveL}
        \mathcal{L}_D (\gamma_{xy}^{-1}(X),\gamma_{xy}^{-1}(Y))=\gamma_{xy}^{-2}(\mathcal{L}_D(X,Y)).
    \end{equation*}\end{itemize}
\end{theorem}
Note that here isometry is intended in the Riemannian sense, i.e.\ $(\gamma_{xy}^{-1})^*g'_y = g_x$, and not in the sense of metric spaces; in other words, distances are not preserved, hence $\gamma_{xy}$ includes possible rescalings. The proof of Theorem \ref{prop_orbits_without_groupoids} is a direct consequence of the following classical result.

\begin{lemma}[\cite{Horn_matrix_analysis}, Corollary 2.5.11]
\label{lem:skew_normal_form}

\begin{enumerate}
\item  Two skew-symmetric matrices $J_1, J_2 \in \mathrm{Mat} (n,\RR)$ are orthogonally similar (i.e.\ there exists an orthogonal matrix $O$ such that $OJ_1O^{-1} = J_2$) if and only if they have the same eigenvalues counted with multiplicities.
\item Given a skew-symmetric matrix $J \in \mathrm{Mat}(n,\mathbb R)$ with non-zero eigenvalues $\pm i\alpha_j$, $j=1,\dots,k$, there exists an orthonormal basis in which $J$ takes the following block-diagonal form
    $$
    J =  
\begin{pmatrix}
0 & 0 \\
0 & \tilde{J}
\end{pmatrix}, 
\qquad 
\text{where }\tilde J =\begin{pmatrix}
0 & \alpha_1  & \dots  & 0 &  0 \\
-\alpha_1  & 0   & \dots & 0 & 0 \\
\vdots & \vdots & \ddots & \vdots   & \vdots \\
0 & 0  & \dots & 0 & \alpha_k \\
0 & 0  & \dots & -\alpha_k  & 0 

\end{pmatrix}.
    $$
    \end{enumerate}
\end{lemma}

\begin{proof}[Proof of Theorem~\ref{prop_orbits_without_groupoids}]

First one notes that if $\gamma_{xy}$ is a Lie algebra isomorphism, it is completely determined by its restriction $\gamma_{xy}^{-1}: D_x \to D_y$. Indeed, almost by definition, $\gamma_{xy}$ is a Lie algebra isomorphism if and only if $\gamma_{xy}^{-1}$ preserves the Levi form, with respect to a $\gamma_{xy}^{-2}: T_xM/D_x \to T_yM/D_y$ as in the statement (which is then necessarily unique). The pair $(\gamma_{xy}^{-1},\gamma_{xy}^{-2})$ is then the Lie algebra isomorphism $\gamma_{xy}$.

Next let us choose orthonormal bases inside $D_x$ and inside $D_y$. This allows us to identify $D_x$ and $D_y$ with $\mathbb R^{\rank D}$, and the maps $J_x$ and $J_y$ from \eqref{eq_map_J} with two skew-symmetric matrices. A different choice of a basis would result in orthogonally similar skew-symmetric matrices. By item 1 of Lemma~\ref{lem:skew_normal_form} two skew-symmetric matrices are orthogonally similar if and only if their eigenvalues together with corresponding multiplicities coincide.
\end{proof}

Another important consequence of Lemma~\ref{lem:skew_normal_form} is the following result on the space of isometries of each $\gr(T_xM)$.

\begin{proposition}\label{prop_description_isotropy_groups}
Let $(M,D,g)$ be a corank-one sub-Riemannian manifold, $x \in M$, and consider the group $\cG_x$ of Lie algebra isomorphisms $\gr(T_xM) \to \gr(T_xM)$ whose component in degree $-1$ is an isometry $(D_x,g_x) \to (D_x,g_x)$.

Let furthermore $m_j^x$, $j=1,\dots k \leq n$ be the multiplicities of the non-zero eigenvalues $\pm i \alpha_j$ of the map $J_x$ from \eqref{eq_map_J}. Then $\cG_x$ is isomorphic to the product
$$
O(\dim\ker J_x)\times U(m_1^x)\times \dots \times U(m_k^x).
$$
\end{proposition}
\begin{proof}
By item 2 of Lemma~\ref{lem:skew_normal_form} we can put $J_x$ into the normal form and see which orthogonal conjugations $OJ_xO^{-1}$ preserve its structure. It is clear that the orthogonal transformation $O$ must preserve $J_x$ when restricted to each invariant subspace corresponding to the eigenvalue $\pm i \alpha^i_x$. Therefore, $\cG_x$ is isomorphic to a subgroup of $O(\dim\ker J_x)\times O(2m_1^x)\times \dots \times O(2m_k^x)$. Consider the invariant subspace that corresponds to a non-zero eigenvalue $\pm \alpha^i_x$. Then $(1/\alpha^i_x)J_x$ restricted to this subspace is the standard symplectic form; therefore, $O$ must preserve it. Hence, we have
$$
\cG_x \cong O(\dim\ker J_x)\times \left(\prod_{j=1}^k O(2 m_j^x) \cap \Sp\left(m_j^x\right) \right) = O(\dim\ker J_x)\times \left(\prod_{j=1}^k U\left (m_j^x\right) \right).
$$
\end{proof}

A direct consequence of the last proposition is that, for generic points of a sub-Riemannian contact manifold of dimension $2n+1$, the groups $\cG_x$ are $n$-dimensional tori.  

\begin{corollary}\label{cor_isotropy_groups_are_tori}
Let $(M,D,g)$ be a contact sub-Riemannian structure such that the type $\alpha: M \to \RR \PP^{n-1}$ \eqref{eq_type_D} takes values in a submanifold of the interior of $\Delta^{n-1}\subset \RR\PP^{n-1}$. Then
$$
\cG_x \cong U(1)\times \dots \times U(1)\cong \SSS^1 \times \ldots \SSS^1 \cong \TT^n.
$$
\end{corollary}

\begin{proof}
Since $D$ is contact, the map $J_x$ is an isomorphism, hence $\dim \ker J_x = 0$. On the other hand, the assumption on $\alpha$, together with Remark \ref{rk_image_type}, shows that all the eigenvalues of $J_x$ are different and with multiplicity $m_i^x = 1$. The conclusion follows from Proposition \ref{prop_description_isotropy_groups}.
\end{proof}

\begin{remark}[Stratification of the simplex]\label{rk_stratification_image_type}
Recall that the $n$-simplex $\Delta^n$ defines a convex polytope: as such, it admits a natural stratification by vertices, edges, faces, etc.; in other words, the $k$-dimensional strata are precisely the $k$-faces of $\Delta^{n}$. In particular, its boundary faces consist of points where either $\lambda_1 = 0$ or where $\lambda_i = \lambda_{i+1}$ for some index $i = 1,\dots,n-1$. For an introduction to stratifications we refer to \cite{Marius_Joao_2018, DK2000}.

The isotropy groups $\cG_x$ described in Proposition \ref{prop_description_isotropy_groups} give an alternative description of this stratification when seen on the embedded image $\Delta^{n-1}\subset \RR \PP^{n-1}$ as in \eqref{eq_embedding_simplex}. More precisely, in the proofs of the previous statements we saw that points in $\Delta^{n-1}$ correspond to normal forms in Lemma~\ref{lem:skew_normal_form}. Two points in $\Delta^{n-1}$ belong to the same stratum if and only if the subgroups of the orthogonal group that preserve the normal form of $J$ are the same. One can read the stratification directly from the eigenvalues $\alpha_i$. The stratum is determined by the multiplicities of the zero eigenvalue and by the multiplicities of the non-zero (ordered) eigenvalues.
\end{remark}

\begin{example}
\label{ex_charlotte_isotropy}
     Let us revisit the Example~\ref{ex_Charlotte} by turning it into into a sub-Riemannian corank-one structure. Assume that we put on that distribution $D \subseteq \RR^5$ a metric, such that all the generators in~\eqref{eq:example_charlotte_generators} form an orthonormal basis. In this case the map $J$ from \eqref{eq_map_J} has the following form:
    $$
    J = \begin{pmatrix}
        0 & 1 & 0 & 0 \\
        -1 & 0 & 0 & 0 \\
        0 & 0 & 0 & |f'(y_2)| \\
        0 & 0 & -|f'(y_2)| & 0 
    \end{pmatrix}.
    $$
    We see that it is already in its normal form up to a reordering of the basis. Since $\dim(M) = 2n+1$ for $n=2$, the type map $\alpha: M \to \RR \PP^1$ is therefore of the form
    $$
    \alpha(x_1,x_2,y_1,y_2,z_2) = \begin{cases}
        [|f'(y_2)|:1] & \text{if } |f'(y_2)|<1\\
        [1:|f'(y_2)|] & \text{if } |f'(y_2)|\geq 1
    \end{cases}
    $$
    The $1$-simplex $\Delta^1$ is just a closed interval; for different choices of $f$, the image of the type map may be any connected subset of it. For example, for $f(y_2)=y_2^2$, the image of the type is the whole $\Delta^1$. The interior points are the ones for which $|f'(y_2)|\notin \{0,1\}$. One boundary of $\Delta^1$ corresponds to $|f'(y_2)|=1$ and the other one to $f'(y_2) = 0$. For the interior points, in accordance with Corollary~\ref{cor_isotropy_groups_are_tori}, the isotropy groups will be isomorphic to tori; on the other hand, for points where $|f'(y_2)|=1$ we will have $U(2)$, and for those where $f'(y_2)=0$ we will have the double cover $O(2)\times U(1)$ of the 2-torus.
\end{example}

\begin{lemma}\label{Lemma_stratum_equinilpotent}If the type map of a corank-one sub-Riemannian structure $(D,g)$ takes values in a single stratum, then the distribution $D$ is equinilpotent (Definition \ref{def_equiregular_equinilpotent_distributions}).
\end{lemma}
\begin{proof}
    The nilpotent model (the isomorphism class of the osculating Lie algebra)  at a point is completely determined by the dimension of the manifold and the dimension of the kernel of the Levi form at the point. The kernel of the Levi form will correspond to an abelian subgroup in the nilpotent model. The restriction of the Levi form to a complement to the kernel will be non-degenerate and, hence, correspond to a Heisenberg Lie algebra. As a consequence, the nilpotent Lie algebra is a direct sum of an abelian Lie algebra and the Heisenberg Lie algebra of the complementary dimension. Since the strata are determined by the multiplicities of eigenvalues of $J$, the multiplicity of the zero eigenvalue, which is also the dimension of the kernel of the Levi form, is constant on each stratum.
\end{proof}

\begin{definition}\label{def_symbols}
Given a sub-Riemannian manifold $(M,D,g)$, its \textbf{sub-Riemannian symbol} at a point $x \in M$ is the pair
\[
(\gr(T_xM),g_x).
\]
We say that $(M,D,g)$ has \textbf{constant symbol} if, for every $x,y \in M$, there is a Lie algebra isomorphism $\gamma_{xy}: \gr(T_x M) \to \gr(T_y M)$ such that the component $\gamma_{xy}^{-1}: (D_x,g_x) \to (D'_y,g'_y)$ is an isometry.
\end{definition}

\begin{example}[sub-Riemannian manifolds with constant symbols]

By Theorem \ref{prop_orbits_without_groupoids}, a sub-Riemannian manifold has constant sub-Riemannian symbol if and only if the type map $\alpha$ from \eqref{eq_type_D} is constant. Therefore, in the rest of this paper we will use the words symbol and type interchangeably. Another characterisation of this property can be found in \cite[Proposition 4.4]{Grong2020}: $(M, D, g)$ has constant symbol if and only if there exists a strongly compatible connection on $\gr (D)$.

Sub-Riemannian manifolds of constant symbol enjoy a variety of wonderful results. In particular, there are efficient ways of finding invariants using methods of Cartan geometry~\cite{sub-lorentzian,Morimoto08}.
\end{example}

\begin{example}[sub-Riemannian manifolds with non-constant symbols]\label{ex_subriemannian_not_constant_symbols}
Consider the family of Examples  \ref{ex_charlotte_isotropy} of sub-Riemannian structures on $\mathbb{R}^5$ depending on a function $f$; as long as $f'$ is non-constant, the result will be a sub-Riemannian structure on $\mathbb{R}^5$ with non-constant symbol.
\begin{enumerate}
\item If we restrict to the half-space $\{y_2>0\}$ and choose, for example, $f(y_2)=\frac{y_2^2}{2}$, then $f'$ is strictly positive, so the resulting sub-Riemannian structure is contact (hence equinilpotent) with points of every possible type for a 5D contact manifold: points in the hyperplane $\{y_2=1\}$ have type $[1:1]$, and any value of the type in the interior of $\Delta^1$ is attained at other points. The type map is not smooth at the points of type $[1:1]$ (the differential of $\alpha$ is not continuous at those points).

\item If $f(y_2)=y_2^3+y_2$ then $f'$ is again strictly positive; the resulting sub-Riemannian structure is contact with points of every possible type: points in the hyperplane $\{y_2=0\}$ have type $[1:1]$; all possible values of the type in the interior of $\Delta^1$ are attained at other points. But in this case $\alpha$ is smooth everywhere (the differential of $\alpha$ is zero at points of type $[1:1]$).

\item If $f(y_2)=e^{y_2}+y_2$ then $f'$ is strictly increasing and takes all values in $(1,+\infty)$, so in this case the type map is a surjective submersion onto the interior of $\Delta^1$. \qedhere
\end{enumerate}
\end{example}

\section{Groupoids and sub-Riemannian structures}\label{section_groupoids_for_subriemannian}

We refer to \cite{MoerdijkMrcun03, Mackenzie05} for the basics on Lie groupoids; we recall below the main concepts and notations which will be relevant in the rest of this paper.

A \textbf{Lie groupoid} will be denoted by $\cG \tto M$, where $\cG$ is the space of arrows and $M$ the space of objects; the source and target maps are denoted, respectively, $s$ and $t$. The \textbf{isotropy group} $\cG_x$ of a point $x \in M$ consists of all arrows $g \in \cG$ such that $s(g) = t(g) = x$; isotropy groups inherit the structure of a Lie group. The \textbf{orbit} $\cO_x$ of a point $x \in M$ consists of all points $y \in M$ such that there exists an arrow $g \in \cG$ between $x$ and $y$; orbits are immersed submanifolds of $M$. Any Lie groupoid $\cG$ has an associated \textbf{Lie algebroid}, sometimes denoted by $\Lie(\cG)$, given by the vector bundle $A := \ker(ds)|_M \to M$, the anchor $\rho := dt|_A: A \to TM$, and the Lie bracket on $\Gamma(A)$ induced by the commutator of right-invariant vector fields on $\cG$.

A Lie groupoid $\cG \tto M$ is called
\begin{itemize}
    \item \textbf{transitive} if there is a single orbit $\cO = M$; equivalently, for every $x,y \in M$, there is an element $g \in \cG$ such that $s(g) = x$ and $t(g) = y$; 
    \item \textbf{regular} if all the orbits have the same dimension; equivalently, the anchor $\rho$ of its Lie algebroid has constant rank;
    \item \textbf{proper} if $(s,t): \cG \to M \times M$ is a proper map (in the topological sense);
\end{itemize}

Fixing any $x \in M$, the $s$-fibre $s^{-1}(x) \subseteq \cG$ has a natural structure of principal $\cG_x$-bundle over the orbit $\cO_x$. If $\cG$ is transitive, all the isotropy Lie groups $\cG_x$ are isomorphic, and all principal $\cG_x$-bundles $s^{-1}(x) \to \cO_x = M$ are isomorphic. This implies the following standard result:
\begin{proposition}\label{prop_transitive_lie_groupoids_principal_bundles}
Any principal $G$-bundle $P \to M$ gives rise to a transitive Lie groupoid, called its \textbf{gauge groupoid}:
\[
\Gauge(P):= (P \times P)/G \tto M.
\]
Conversely, any transitive Lie groupoid $\cG \tto M$ arises as the gauge groupoid of the principal $\cG_x$-bundle
\[
P := s^{-1} (x) \xrightarrow{t|_P} M,
\]
where $x$ is any point in $M$. This yields a 1-1 correspondence between (isomorphism classes of) transitive Lie groupoids and (isomorphism classes of) principal bundles.

Furthermore, given a principal $G$-bundle $P \to M$, any representation of $G$ on a vector space $V$ induces a representation of the gauge groupoid $\Gauge(P) \tto M$ on the associated vector bundle
\[
P[V] := (P \times V)/G \to M.
\]
Conversely, any representation of a transitive groupoid $\cG \tto M$ on a vector bundle $E \to M$ arise in this way from the induced representation of the isotropy group $\cG_x$ on the fibre $V := E_x$, where $x$ is any point in $M$. 
\end{proposition}

\begin{example}[frame bundles and general linear groupoids]\label{ex_frame_bundle_general_linear_groupoid}
The \textbf{frame bundle of a rank $k$-vector bundle $E \to M$} is the principal $\GL(k)$-bundle $\Fr(E) \to M$ given by linear isomorphisms $\RR^k \to E_x$, for any $x \in M$. Its associated gauge groupoid is isomorphic to the \textbf{general linear groupoid} $\GL(E) \tto M$, i.e.\ the (transitive) Lie groupoid consisting of all linear isomorphisms $\gamma: E_x \to E_y$, for any $x,y \in M$. The source and target of $\gamma$ are, respectively, $x$ and $y$, while the multiplicative is given by the composition. In the following we will always consider the case where $E$ is the tangent bundle $TM \to M$ of an $n$-manifold $M$, so that $\Fr(M):= \Fr(TM)$ is a principal $\GL(n)$-bundle, called the \textbf{frame bundle of $M$}.

Furthermore, given a Lie subgroup $G \subseteq \GL(n)$, the associated gauge groupoid of any \textbf{$G$-structure} $P \subseteq \Fr(M)$ on $M$ (i.e.\ a $G$-reduction of the frame bundle) is isomorphic to the subgroupoid $\cG \subseteq \GL(TM)$ consisting of linear isomorphisms preserving the corresponding linear $G$-structures on the tangent spaces. For instance, if $P \to M$ was the orthonormal frame bundle (i.e.\ the $O(n)$-structure associated to a Riemannian metric on $M$), then $\cG \subseteq \GL(TM)$ would consist of linear isometries $(T_x M, g_x) \to (T_y M, g_y)$.

As per Proposition \ref{prop_transitive_lie_groupoids_principal_bundles}, notice also that the natural representation of $G \subseteq \GL(n)$ on $\RR^n$ induces a representation of its gauge groupoid on the associated vector bundle $P[\RR^n]$, which is canonically isomorphic to $TM$.
\end{example}

We recall now, for future use, 
a couple of important concepts associated to the objects discussed in Example \ref{ex_frame_bundle_general_linear_groupoid}. They will not be used elsewhere in Section \ref{section_groupoids_for_subriemannian}, so the reader can also safely skip this part and go back to it when reading Section \ref{section_Cartan_connections}.

The frame bundle $\pi: \Fr(M) \to M$ is endowed with a distinguished vector-valued form, called the \textbf{tautological form} (or soldering form, or canonical form) \cite{Sternberg64, Kobayashi95}:
\begin{equation}\label{eq_tautological_form_frames}
\theta_{\taut} \in \Omega^1 (\Fr(M), \RR^n), \quad \quad (\theta_{\taut})_p (v) := p^{-1} (d\pi(v)),
\end{equation}
where $p: \RR^n \to T_x M$ is a frame of $M$. Note that $\theta_{\taut}$ satisfies the following properties:
\begin{itemize}
\item $\theta_{\taut}$ is pointwise surjective;
\item $\theta_{\taut}$ is $\GL(n)$-equivariant (with respect to the natural representation of $\GL(n)$ on $\RR^n$);
\item $\ker(\theta_{\taut}) = \ker(d\pi)$. 
\end{itemize}

On the other hand, the groupoid $\GL(TM) \tto M$ is endowed with a distinguished vector bundle-valued form, sometimes also called \textbf{tautological form}:
\begin{equation}\label{eq_tautological_form_GL}
\omega_{\taut} \in \Omega^1 (\GL(TM), t^*TM), \quad \quad (\omega_{\taut})_{\gamma_{xy}} (v) := dt (v) - \gamma_{xy} (ds (v)),
\end{equation}
where $\gamma_{xy}: T_x M \to T_y M$ is an element of $\GL(TM)$. The form $\omega_{\taut}$ satisfies as well several properties, analogous to those of $\theta_{\taut}$:
\begin{itemize}
    \item $\omega_{\taut}$ is pointwise surjective;
    \item $\omega_{\taut}$ is multiplicative (with respect to the natural representation of $\GL(TM)$ on $TM$):
\begin{equation}\label{eq_multiplicativity_GL}
(m^* \omega_{\taut})_{(g,h)} = (\pr_1 ^* \omega_{\taut})_{(g,h)} + g \cdot (\pr_2^* \omega_{\taut})_{(g,h)} \text{ for every } (g,h) \in \GL(TM)^{(2)};
\end{equation}
    \item $\omega_{\taut} \cap \ker(ds) = \omega_{\taut} \cap \ker(dt) = \ker(ds) \cap \ker(dt)$.
\end{itemize}
From a different perspective, under the identification $\GL(TM) \cong J^1 (M,M)$, the form $\omega_{\taut}$ coincides with the standard Cartan form on jet bundles \cite{Saunders89}. Furthermore, we remark that $(\GL(TM),\omega_{\taut})$ is an example of a Pfaffian groupoid (see Remark \ref{rk_relation_Pfaffian_groupoids} later on).

The similarities between $\theta_{\taut}$ and $\omega_{\taut}$ (in both their definitions and their properties) are not coincidental. In fact, $\omega_{\taut}$ is precisely the form corresponding to $\theta_{\taut}$ via a standard general construction which allows one to transport $V$-valued forms on principal bundles $P$ to multiplicative $P[V]$-valued form on their gauge groupoids $\Gauge(P)$ (in the case above, the associated bundle $\Fr(M)[\RR^n]$ is canonically isomorphic to $TM$, as mentioned at the end of Example \ref{ex_frame_bundle_general_linear_groupoid}). This is also a particular instance of a more general construction involving principal groupoid bundles and Morita equivalences; see e.g.\ the appendix of \cite{Cattafi21} for a brief overview of these results and further references.

\begin{remark}[restriction of the tautological form]\label{rk_restriction_tautological_form_GL_TM}

Any \textit{transitive} Lie subgroupoid $\cG \subseteq \GL(TM)$ inherits a 1-form $\omega_{\taut} \in \Omega^1 (\cG, t^* TM)$ by restricting the form $\omega_{\taut} \in \Omega^1 (\GL(TM), t^* TM )$ from \eqref{eq_tautological_form_GL}. Note that all the properties of $\omega_{\taut}$ are preserved by this restriction; in particular, the transitivity hypothesis is needed to ensure that $\omega_{\taut}$ remains pointwise surjective after restricting it to $\cG$.

On the other hand, if $\cG$ is not transitive, the restriction of $\omega_{\taut}$ to $\cG$ is not pointwise surjective. Indeed, its image of the restriction is precisely the (involutive) subbundle $\Ima(\rho) \subseteq TM$, i.e.\ the tangent distribution to the orbit foliation of $\cG$, where $\rho: A \to TM$ is the anchor of the Lie algebroid $A = \Lie(\cG)$. To show this, one can first check that the multiplicativity of $\omega_{\taut}$ implies $(\omega_\taut)_g (T_g \cG) = (\omega_\taut)_{1_{t(g)}} (\ker (d_{1_{t(g)}} s) )$ for every $g \in \cG$ and $T_{1_x}\cG = \ker(d_{1_x} s) + \ker((\omega_\taut)_{1_x})$ for every $x \in M$. The formula \eqref{eq_tautological_form_GL} yields then
\[
(\omega_\taut)_{1_x} (T_{1_x} \cG) = (\omega_\taut)_{1_x} (\ker(d_{1_x} s)) = d_{1_x} t (\ker(d_{1_x} s)) = \rho_x (A_x) = \Ima(\rho_x).
\]

In general, to obtain a pointwise surjective form (whose importance will be clearer in Section \ref{section_cartan_connections_subriemannian_contact_manifolds}),  one would need to restrict its coefficient bundle $TM \to M$ to the subbundle $\Ima(\rho) \to M$. However, the natural representation of $\GL(TM)$ on $TM$ does not necessarily restrict to a representation of $\cG$ on $\Ima(\rho)$: indeed, an arbitrary element $\gamma: T_x M \to T_y M$ of $\cG \subseteq \GL(TM)$ does not necessarily send the tangent space to the orbit of $x$ to the tangent space to the orbit of $y$. In conclusion, by restricting the coefficients of $\omega_{\taut}|_\cG$ to $\Ima(\rho)$, one obtains a form 
\[
\omega_{\taut}|_\cG \in \Omega^1 (\cG, t^* \Ima(\rho)),
\]
with the same properties of the original tautological form on $\GL(TM)$, including pointwise surjectivity, but possibly excluding multiplicativity (since the expression $g \cdot$ in the equation \eqref{eq_multiplicativity_GL} makes no longer sense). We will discuss in Example \ref{ex_nontransitive_subgroupoid_GL_as_Cartan_groupoids} how to address this issue.
\end{remark}

\subsection{Groupoids and distributions}

\begin{proposition}\label{prop_GL_gr_bracket-generating_equiregular}
 Let $D \subseteq TM$ be a bracket-generating distribution (Definition \ref{def_equiregular_equinilpotent_distributions}) and consider
 \begin{align*}
\GL(\gr(TM)) \supset \GL^{\gr} (D) &:= \{ \gamma_{xy}: \gr(T_xM) \to \gr(T_yM) \text{ isomorphisms of graded vector spaces} \}_{x,y \in M} \\
&= \{ (\gamma^i_{xy}: D^i_x/D^{i+1}_x \to D^i_y/D^{i+1}_y)_i \text{ linear isomorphisms } \}_{x,y \in M}.
 \end{align*}
 Then $\GL^{\gr} (D)$ is a groupoid over $M$. Furthermore, if $D$ is also equiregular, $\GL^{\gr} (D) \tto M$ is a transitive Lie groupoid.
\end{proposition}

In the particular case of corank-one distributions $D \subseteq TM$, an element $\gamma \in \GL^\gr(D)$ is given by two isomorphisms:
 \[
  \gamma^{-1}_{xy}: D_x \to D_y, \quad \quad \gamma^{-2}_{xy}: T_xM/D_x \to T_yM/D_y.
 \]
Note that $\GL^{\gr}(D)$ can be defined (and is a groupoid) also for more general distributions, as long as, for any $x \in M$, the sequence $\{D^i_x\}_{i=-\infty}^{-1}$ stabilises after a finite number of steps (even if not necessarily to $T_x M$).

\begin{proof}
 The groupoid structure of $\GL^\gr(D)$ is given by
 \[
  s(\gamma_{xy}) = x, \quad \quad t (\gamma_{xy}) = y,
 \]
 with multiplication obtained by concatenating the maps $\gamma^i_{xy}$; the unit at $x$ is given by the identity of $\gr(T_x M)$, and the inversion map is obtained by inverting the isomorphisms.

 If $D$ is equiregular, then
 \begin{equation}\label{eq_graded_frames}
  \Fr^\gr(D):= \{ (\phi^i_x: D_x^i/D_x^{i+1} \to \RR^{\rank(D^i) - \rank (D^{i+1})})_i \text{ linear isomorphisms}  \}_{x \in M}
 \end{equation}
is a smooth manifold. Moreover, it is a principal bundle and its associated gauge groupoid is isomorphic to $\GL^\gr(D)$, which becomes therefore a Lie groupoid. In other words, $\GL^\gr(D)$ inherits its groupoid structure from the general linear groupoid $\GL(E) \tto M$ (Example \ref{ex_frame_bundle_general_linear_groupoid}), for $E$ the vector bundle $\gr(TM) = \oplus_{x \in M} \gr (T_x M) \to M$.

Alternatively, one can show that $\GL^\gr(D)$ is canonically isomorphic to the fibred product $\GL (D^{-1}/D^{-2})\times_{(s,t)}\ldots\times_{(s,t)}\GL(TM/D^{-k+1})$ of transitive Lie groupoids over $M$.
 \end{proof}

\begin{example}
 For $D = TM$ (which is trivially bracket-generating and equiregular) one recovers the ordinary (transitive) Lie groupoid $\GL(TM) \tto M$, which is indeed the gauge groupoid of the standard frame bundle $\Fr(M) \to M$ (Example \ref{ex_frame_bundle_general_linear_groupoid}).
\end{example}

As in the discussion following Example \ref{ex_frame_bundle_general_linear_groupoid}, we recall now, for future use, how the tautological form \eqref{eq_tautological_form_GL} can be adapted to $\GL^\gr(D)$. This will not be used elsewhere in Section \ref{section_groupoids_for_subriemannian}, so the reader can also safely skip this part and go back to it when reading Section \ref{section_Cartan_connections}.

Consider a bracket-generating equiregular distribution $D \subseteq TM$ (in the following, for simplicity we will assume that it is of step 2, but everything would work in full generality) and the vector bundle
\[
\gr(TM) = \coprod_{x \in M} \gr(T_x M) \to M.
\]
For any Lie groupoid $\cG \tto M$, the filtration $\{ D^i\}_i$ given by the distribution $D$ induces a natural filtration $\{T^i \cG := (ds)^{-1} (D^i) \cap (dt)^{-1} (D^i) \}_i$ of its tangent space
\begin{equation}\label{eq_filtration_groupoid}
T^0 \cG := \ker(ds) \cap \ker(dt) \subseteq T^{-1} \cG := (ds)^{-1} (D) \cap (dt)^{-1} (D) \subseteq T^{-2} \cG := T\cG,
\end{equation}
as well as a filtration $\{ A^i := T^i \cG \cap A\}_i$ of its Lie algebroid $A = \Lie(\cG)$
\begin{equation}\label{eq_filtration_algebroid}
A^0 := \ker(\rho) \subseteq A^{-1} := \rho^{-1}(D) \subseteq A^{-2} := A,
\end{equation}
and a filtration $\{ \Ima(\rho)^i := \rho (A^i) = \Ima(\rho) \cap D^i \}_i$ on the image of the anchor $\Ima(\rho) \subseteq TM$
\begin{equation}\label{eq_filtration_image_rho}
\Ima(\rho)^{-1} := \rho(\rho^{-1}(D)) = \rho(A) \cap D \subseteq \Ima(\rho)^{-2} := \rho(A).
\end{equation}

In particular, the (transitive) Lie groupoid $\cG = \GL^{\gr} (D) \tto M$ is naturally endowed with the following 1-form
\begin{equation}\label{eq_tautological_form_GL_gr_TM}
\omega_{\taut} \in \Omega^1 (\cG, t^* \gr(TM) ), \quad \omega_{\gamma}(v) := \begin{cases}
    dt(v) - \gamma^{-1}(ds(v)) \in D_{t(\gamma)} \quad \quad \text{if } v \in T^{-1} \cG, \\
    [dt(v)] - \gamma^{-2} [ds(v)] \in T_{t(\gamma)}M/D_{t(\gamma)} \quad \text{otherwise},
\end{cases}
\end{equation}
which can be equivalently described by the collection of smooth vector bundle maps
\[
\omega^{-2}_{\taut}: T \cG \to t^*TM/D, \quad v \mapsto [dt(v)] - \gamma^{-2} [ds(v)],
\]
\[
\omega^{-1}_{\taut}: T^{-1} \cG \to t^*D, \quad v \mapsto dt(v) - \gamma^{-1}(ds(v)) \in D_{t(\gamma)}.
\]
It is easy to see that $\omega_{\taut}$ plays the same role as the tautological form $\omega_{\taut} \in \Omega^1 (\GL(TM), t^*TM)$ from \eqref{eq_tautological_form_GL}. In particular, it satisfies the following properties:
\begin{itemize}
    \item it is pointwise surjective;
    \item it is multiplicative (with respect to the natural representation of $\GL^\gr(D)$ on $\gr(TM)$):
\begin{equation}\label{eq_multiplicativity_GL_gr}
(m^* \omega_{\taut})_{(g,h)} = (\pr_1 ^* \omega_{\taut})_{(g,h)} + g \cdot (\pr_2^* \omega_{\taut})_{(g,h)} \text{ for every } (g,h) \in \GL^\gr(D)^{(2)};
\end{equation}
    \item $\ker(\omega_{\taut}^{-2}) \cap \ker(ds) = \ker(ds) \cap (dt)^{-1}(D)$ and $\ker(\omega_{\taut}^{-2}) \cap \ker(dt) = (ds)^{-1}(D) \cap \ker(dt)$;
    \item $\ker(\omega_{\taut}^{-1}) \cap \ker(ds) = \ker(\omega_{\taut}^{-1}) \cap \ker(dt) = \ker(ds) \cap \ker(dt)$.
\end{itemize}
Unlike $(\GL(TM), \omega_\taut)$, the pair $(\GL^\gr (D), \omega_\taut)$ is not a Pfaffian groupoid, since the distribution
\[
( \ker(\omega_{\taut}) \cap \ker(ds) )|_A = (\ker(ds) \cap (dt)^{-1}(D) ) |_A
\]
is not involutive (see Remark \ref{rk_relation_Pfaffian_groupoids}).

\begin{remark}[restrictions of the tautological form]\label{rk_restriction_tautological_form_GL_gr_TM}

When restricting the tautological form $\omega_{\taut} \in \Omega^1 (\GL^{\gr} (D), t^* \gr(TM) )$ from \eqref{eq_tautological_form_GL_gr_TM} to any Lie subgroupoid $\cG \subseteq \GL^{\gr} (D)$, one faces the same issues discussed in Remark \ref{rk_restriction_tautological_form_GL_TM}.

If $\cG$ is \textit{transitive}, everything works nicely, and the restriction of $\omega_{\taut}$ to $\cG$ satisfies all its original properties. On the other hand, if $\cG$ is not transitive, pointwise surjectivity fails. Indeed, with the same argument as in Remark \ref{rk_restriction_tautological_form_GL_TM}, one shows that the image of $\omega_\taut$ is
\[
\Ima(\omega_\taut) = \rho(A^{-1}) \oplus \rho(A)/\rho(A^{-1}) = \gr (\Ima(\rho)),
\]
i.e.\ the graded version (with respect to the filtration \eqref{eq_filtration_image_rho}) of the image $\Ima(\rho) \subseteq TM$ of the Lie algebroid $A = \Lie(\cG)$ by the anchor $\rho: A \to TM$.

In general, to obtain a pointwise surjective form (whose importance will be clearer in Section \ref{section_cartan_connections_subriemannian_contact_manifolds}), one would need to restrict its coefficient bundle $\gr(TM) \to M$ to the subbundle $\Ima(\rho) \subseteq \gr(TM)$, obtaining a form:
\[
\omega_{\taut}|_\cG \in \Omega^1 (\cG, t^* \gr(\Ima(\rho))),
\]
with the same properties of the original tautological form on $\GL^{\gr} (D)$, excluding multiplicativity. The reason is again that the natural representation of $\GL^{\gr} (D)$ on $\gr(TM)$ does not necessarily restrict to a representation of $\cG$ on $\gr(\Ima(\rho))$, hence the expression $g \cdot$ in the equation \eqref{eq_multiplicativity_GL_gr} does not necessarily make sense. We will discuss in Example \ref{ex_nontransitive_subgroupoid_GL_gr_as_Cartan_groupoids} how to address this issue.
\end{remark}

\subsection{Groupoids and sub-Riemannian structures}

\begin{proposition}\label{prop_definition-G}
Let $(M^{2n+1},D,g)$ be a corank-one sub-Riemannian manifold. Then
\begin{equation}\label{eq_Lie_groupoid_subriemannian_symmetries}
 \cG := \left\{ \gamma_{xy} \in \GL^{\gr} (D) |
 \begin{array}{l}\gamma_{xy}: \gr(T_xM) \to \gr(T_yM) \text{ is a Lie algebra isomorphism and }\\
\gamma_{xy}^{-1}: (D_x,g_x) \to (D_y,g_y) \text{ is an isometry }
\end{array} \right\}
\end{equation}
is a (set-theoretical) subgroupoid $\cG\tto M$ of $\GL^{\gr}(D)$. Its orbits are the fibres of the type map $\alpha$.
\end{proposition}

\begin{proof} $\cG$ is clearly a subgroupoid of $\GL^{\gr}(D)$. That its orbits are the fibres of $\alpha$ is a rephrasing of Theorem \ref{prop_orbits_without_groupoids}. 
\end{proof}

We remark immediately that, unlike $\GL^\gr(D)$, the groupoid $\cG$ is not necessarily transitive.

\begin{example}
A sub-Riemannian structure ($M,D,g)$ has constant symbol (Definition \ref{def_symbols}) if and only if the groupoid $\cG$ \eqref{eq_Lie_groupoid_subriemannian_symmetries} is transitive. In that case, $\cG$ is smooth as well, and it is the gauge groupoid of the bundle of "sub-Riemannian frames" (on top of which one can build a canonical Cartan geometry) from  \cite{Morimoto93, Morimoto08, HongMorimoto24}.
\end{example}

\begin{remark}
    (Equivalent description of $\mathcal{G}$, as a subgroupoid of $\GL(D)$) As seen already in the proof of Theorem \ref{prop_orbits_without_groupoids}, since an arrow $\gamma_{xy}\in \mathcal{G}$ is a Lie algebra isomorphism,  it is completely determined by its restriction $\gamma_{xy}^{-1}: D_x \to D_y$, and the properties that it is both an isometry and preserves the Levi map. By $\gamma_{xy}^{-1}$ preserving the Levi map we mean that there exists a $\gamma_{xy}^{-2}: T_xM/D_x \to T_yM/D_y$ (which is then necessarily unique) such that
    \begin{equation}
        \mathcal{L}_D (\gamma_{xy}^{-1}(X),\gamma_{xy}^{-1}(Y))=\gamma_{xy}^{-2}(\mathcal{L}_D(X,Y)),
    \end{equation} for all $X,Y\in D_x$.
     Therefore, if $(D,g)$ is a corank-one sub-Riemannian structure, we can think of $\cG$ as the subgroupoid $\GL(D,g,\mathcal{L}_D)\subset \GL(D)$ consisting of those isometries of $(D,g)$ which additionally preserve the Levi map.
     
     In the coorientable case $D=\ker(\theta)$, the Lie groupoid $\GL(D,g,\mathcal{L}_D)$ is also isomorphic to $\GL(D,g,\omega)$,  consisting of those isometries of $(D,g)$ which preserve the $2-$form $\omega := d\theta|_{D}$.
\end{remark}

Recall that, for any function $\alpha: M \to N$, the set
 \[
  M \times_\alpha M := \{ (x,y) | \alpha(x) = \alpha(y) \} \subseteq M \times M
 \]
 is a subgroupoid over $M$ of the pair groupoid $M \times M \tto M$. For any groupoid $\cH \tto M$, its anchor map
 \[
  (s,t): \cH \tto M \times M
 \]
is a groupoid morphism. In the case of $\cG$ from \eqref{eq_Lie_groupoid_subriemannian_symmetries} the image of $(s,t)$ is precisely $M \times_\alpha M$. When $\alpha$ is a submersion between smooth manifolds, $M \times_\alpha M$ is a proper regular Lie groupoid, called the submersion groupoid of $\alpha$. That is the setting for the following proposition.

\begin{theorem}\label{prop_G_regular_Lie_subgroupoid}
Assume that the image $\alpha(M)$ of the type $\alpha: M \to \RR \PP^{n-1}$ is a submanifold of a single stratum of $\Delta^{n-1}$ (see Remark \ref{rk_image_type}) and that $\alpha$ is a submersion onto $\alpha(M)$; then $\cG$ is a regular Lie subgroupoid of $\GL^\gr(D)$.
\end{theorem}

\begin{proof}
We show that $\cG $ is a smooth manifold by constructing a smooth atlas. Consider an atlas $(U_i,\varphi_i)$ of $M$. Choose a one-form $\theta_i$ for each chart such that $D|_{U_i} = \ker \theta_i$ and construct the operators $J_i$ from \eqref{eq_map_J}. Since $\alpha$ takes values in a stratum of $\Delta^{n-1}\subset \RR\PP^n$, the eigenspaces of $J_i$ depend smoothly on the points of the chart and form a locally trivial bundle. Also, the isotropy groups depend only on the dimensions of the eigenspaces of non-zero eigenvalues, and not on their exact values. Therefore, by the assumptions of the proposition, all the isotropy groups are isomorphic to a single Lie group, that we call $G$. 

Choose a reference basis of $\Gamma(\gr(D))$ adapted to the decomposition into the eigenspaces of $J_i$, and whose elements are in addition orthonormal on $D^{-1}$. A chart of $\cG$ is then constructed as
$$
(U_i\times_\alpha U_j) \times G
$$
Indeed, there exists an arrow between $x\in U_i$ and $y\in U_j$ if and only if $\alpha(x) = \alpha(y)$. Since we have chosen an adapted frame at $\gr(T_xM)$ and an adapted frame at $\gr(T_yM)$, we can write down the matrix of the map $\gamma_{xy}$, which will be an element of a faithful representation of $G$. Geometrically we can think of it this way: consider the image of the frame at $x$ under the map $\gamma_{xy}$. Then there is a unique element of $\cG_x \cong G$ which will transform this image into the fixed frame at $y$. Since $\alpha$ is a submersion, we have that $U_i\times_\alpha U_j$ is smooth and gives rise to a smooth structure on our charts. It remains only to see if the transition maps are smooth at the intersections. 

If we take an intersection between two charts we have two different choices of frames that we used in our construction. But the two frames are connected via a gauge transformation. Namely at the intersection of two charts we have a map $f_{ij} \in C^\infty(U_i\cap U_j,G)$, such that the passage between the choice of the adapted frame on $U_i$ and that on $U_j$ at $x\in U_i\cap U_j$ is represented by the matrix associated to $f_{ij}(x)$. Similarly for $y\in U_k\cap U_l$ we will have a function $f_{kl} \in C^\infty(U_i\cap U_j,G)$. Then, when we pass from one chart to another, the variables $g\in G$ change as $g\mapsto f_{kl}(y) g f_{ij}(x)^{-1}$, which is smooth. This proves that $\cG$ is a smooth manifold.

It remains to see that all the structure maps are smooth, and that source and target are submersions. It is straightforward to check that the charts that we have constructed allow us to identify locally $\cG$ with the product of a submersion groupoid and a Lie group $(M\times_\alpha M) \times G$. Since smoothness of structure maps is a local property and they are smooth on the latter, they will also be smooth on $\cG$. Similarly, the source and target maps $\cG \to M$ are locally given by those of a submersion groupoid, so they are submersions. We conclude that $\cG$ is a Lie groupoid.

Finally, by hypothesis and using Lemma \ref{Lemma_stratum_equinilpotent}, $D$ is equinilpotent and in particular equiregular. So all isotropy groups of $\GL^\gr(D)$ are isomorphic to a single Lie group $\GL^\gr(\mathfrak{n})$, with $G$ being a Lie subgroup of $\GL^\gr(\mathfrak{n})$. The chart domains $U_i\times_\alpha U_j \times G$ for $\cG$ can easily be seen as submanifolds of similar chart domains $U_i\times_\alpha U_j \times \GL^\gr(\mathfrak{n})$ for $\GL^\gr(D)$.  Therefore, besides being a set-theoretical subgroupoid, $\cG$ is a submanifold, and hence a Lie subgroupoid of $\GL^\gr(D)$.
\end{proof}

\begin{remark}[Orbit foliation of $\cG$]\label{rmk_orbit_foliation_G}
The image $\Ima(\rho) \subseteq TM$ of the anchor $\rho$ of any Lie algebroid $\Lie(\cG)$ coincides with the (possibly singular) involutive distribution associated to the (possibly singular) foliation by the orbits of $\cG$.
In the setting above the orbit foliation of $\cG$ is the simple foliation associated to the submersion $\alpha$. The associated (regular) involutive distribution is $\ker(d\alpha) \subseteq TM$. In other words, one has $\Ima(\rho) = \ker(d\alpha)$.
\end{remark}

\begin{proposition}\label{prop_G_proper}
The groupoid $\cG$ is proper.
\end{proposition}
\begin{proof}
All fibres $(s,t)^{-1}(x,y)$ of the map $(s,t):\cG\to M\times M$ are compact, since they are diffeomorphic to isotropy groups of $\cG$, which are compact by Proposition  \ref{prop_description_isotropy_groups}. The image of $(s,t)$ is closed, being an equivalence relation $M\times_\alpha M\subset M\times M$ which has a Hausdorff quotient space, $\alpha(M)$. 

The map $(s,t):\cG\to M\times M$ is therefore proper, being a map between locally compact Hausdorff spaces with compact fibres and closed image.
\end{proof}

\begin{remark}[On structures realizing a given type]\label{rmk_realization_problem}
It is not hard, as seen in Examples \ref{ex_subriemannian_not_constant_symbols}, to find corank-one sub-Riemannian structures having diverse images of the type in the simplex.
Asking that the type map $\alpha:M\to \Delta^{n-1}\subset \RR \PP^{n-1}$ is a submersion imposes some restrictions on the manifolds, and on the structures we are considering, but there are still many examples, namely in local models.

The question becomes harder, however, in case we would like to find \textit{compact} corank-one sub-Riemannian manifolds with a given type image. Of course, in that case, asking that the type is a submersion into its image requires that the image is a compact submanifold of $\Delta^{n}$.
The simplest example (of non-constant type) where that can happen is for $n=3$. In that case it would be interesting to find (and possibly classify?) compact corank-one sub-Riemannian manifolds whose image under the type map is a circle inside the open stratum of $\Delta^2$, and whose type map is a submersion onto the circle.

Finally, it would of course be interesting to know when is it possible that $\cG$ is a Lie groupoid if $\alpha$ is not a submersion.
\end{remark}

\begin{remark}[Gerbes]\label{rk_gerbes}
In the conditions of Theorem \ref{prop_G_regular_Lie_subgroupoid}, the Lie groupoid $\cG$ fits in a short exact sequence of Lie groupoids

\[\begin{tikzcd}\label{extensionG}
	1 & I(\cG) & \cG & M \times_\alpha M & 1,
	\arrow[from=1-1, to=1-2]
	\arrow[hook, from=1-2, to=1-3]
	\arrow[two heads, from=1-3, to=1-4]
	\arrow[from=1-4, to=1-5]
\end{tikzcd}\]
where $I(\cG)$ is the bundle of isotropy groups of $\cG$ (locally trivial, with fibre the Lie group $G$). This is called a $G$-extension over $M \times_\alpha M$, and its Morita equivalence class is a \textbf{$G$-gerbe} over the manifold $B=\alpha(M)$, as in \cite[Section 3.4]{LGSX09}. 

When $D$ is a contact distribution and $\alpha$ is a submersion onto a submanifold $B$ of the interior of $\Delta^{n-1}$, then $I(\cG)$ is a torus bundle $\mathcal{T}$, by Corollary \ref{cor_isotropy_groups_are_tori}. If the extension is central then it defines a $\mathcal{T}$-gerbe over $B$, where $\mathcal{T}$ is a torus bundle over $B$. Such gerbes are classified by their Dixmier-Douady class $c_2(\cG)\in H^2(B,\underline{\mathcal{T}})$ \cite[Theorem 8.2.6]{PMCT2}, where $\underline{\mathcal{T}}$ denotes the sheaf of sections of $\mathcal{T}$.

As part of the realization problems associated with the type map $\alpha$, it would be interesting to understand which $\mathcal{T}$-gerbes can be realised in this way by a sub-Riemannian contact manifold, and when two sub-Riemannian contact manifolds (necessarily of the same dimension) give rise to the same $\mathcal{T}$-gerbe.
\end{remark}

\subsection{Groupoids preserving the type foliation}

Consider the subgroupoid $\cG'$ of $\cG$ defined as
\begin{equation}\label{eq_Lie_groupoid_subriemannian_symmetries_reduced}
\cG' := \left\{ \gamma_{xy} \in \cG | \gamma_{xy}(\gr(\ker(d\alpha))) \subseteq \gr(\ker(d\alpha)) \right\}.
\end{equation}
This is a groupoid whose arrows in addition to being isometries between nilpotent Lie algebras $(\gr(T_xM),g_x)$ and $(\gr(T_yM),g_y)$ also map $\gr(\ker d_x\alpha) \subset \gr(T_x M)$ to $\gr(\ker d_y\alpha) \subset \gr(T_y M)$. In order to understand the geometry of this groupoid, first we need a better understanding of the geometry of $\ker d\alpha$.

\begin{lemma}\label{lemm:transversality_of_kerdalpha}
Let $(M,D,g)$ be a corank-one sub-Riemannian manifold. Assume that the image $\alpha(M)$ of the type $\alpha: M \to \RR \PP^{n-1}$ is a submanifold of a single stratum of $\Delta^{n-1}$ and that $\alpha$ is a submersion onto $\alpha(M)$.
Then $\ker d\alpha$ is transversal to $D$ at every point.
\end{lemma}

\begin{proof}
    We separate the argument into two cases: when $\dim M$ is odd and even. We will prove only the odd case, because the even case is proved similarly. 
    
    Suppose that $\dim M = 2n +1$. Since the statement is local we assign around $x$ a defining one-form $\theta$ such that $D = \ker\theta$. 
    Denote $m^x_0 = \dim \ker J_x$ the multiplicity of the eigenvalue zero. Since we assume that $\alpha$ takes values in a fixed stratum, $m^x_0\equiv m_0$ is a constant function. This implies that the dimension of the image of $\alpha$ can be at most $n-m_0-1$. Therefore, 
    $$
    \dim \ker d\alpha =2n+1 - \dim \Ima \alpha\geq  2n+1 -(n-m_0-1) = n+m_0+2.
    $$
    Recall that the value of $d\theta (X,Y)$ measures how far the commutator $[X,Y]$, $X,Y\in \Gamma(D)$ is from being inside $D$. Under the assumptions $\ker d\alpha$ is an involutive distribution. Therefore, if by contradiction $\ker d_x\alpha \subset D_x$ (only possibility other than transversality with $D$, because $D_x$ is corank one), then $\ker d_x\alpha$ is an isotropic subspace. Restrict the form $d\theta$ to the orthogonal complement to $\ker d\theta = \ker J \subset D$. By construction it must be non-degenerate and, therefore, symplectic. As result, the dimension of a maximally isotropic subspace of $d\theta$ must be $\dim \ker J + (2n-\dim \ker J)/2 = n + m_0/2$. Note also that since $\dim D_x$ is even in this case, $m_0$ is also even, hence we have a well defined integer. But
    $$
    n+m_0+2 > n + m_0/2,
    $$
    which means that $\ker d\alpha$ is not isotropic and, as a consequence, not involutive. 
\end{proof}

\begin{corollary} In the conditions of Lemma \ref{lemm:transversality_of_kerdalpha}, if we restrict a corank-one sub-Riemannian distribution to the orbits of $M \times_\alpha M$, it will be a corank-one sub-Riemannian structure.    
\end{corollary}

The previous lemma implies that (under its assumptions) if $\gamma_{xy} \in \cG$, then $\gamma_{xy}$ preserving $\gr(\ker d\alpha)$ is equivalent to $\gamma_{xy}^{-1}(\ker d_x\alpha \cap D_x) \subset \ker d_y\alpha \cap D_y$. Let us make the following construction to characterise when that happens.

Let $P$ be the reduction of $\Fr^\gr(D)$ from \eqref{eq_graded_frames} consisting of orthonormal frames of $D$ adapted to the eigenspaces of $J$. It is a principal bundle over $M$ with structure group $G\subset \GL(d)$, where $d=\dim D$; under the assumptions, $G$ is isomorphic to the isotropy group of $\cG$ at any point $x\in M$. 

The group $\GL(d)$ acts naturally on $\Gr(k,d)$, the Grassmannian of $k$-subspaces of $\mathbb R^d$. Consider the restriction of this action to the group $G$ and construct $F_k=P\times_G \Gr(k,d)$, the bundle associated with $P$ with fibre $\Gr(k,d)$. The projection map $\pi: \Gr(k, d)\to \Gr(k, d)/G$ induces a well defined map $f: F_k\to \Gr(k, d)/G$, $[p,V]\mapsto [V]$.

Let $(U_i)_{i\in I}$ be an open cover of $M$ and choose a local section  $(p_i)$ of $P$ over each $U_i$; and recall that $p_i(x):D_x\to \mathbb R^d$ is a linear isomorphism. Define local sections $s_i$ of $F_k$ over $U_i$ by $s_i(x)=[p_i(x),p_i(x)^{-1}(\ker d_x \alpha  \cap D_x)]$, where $\dim (\ker d_x \alpha  \cap D_x)=k$.

\begin{definition}
    The \textbf{secondary type} map of a corank-one sub-Riemannian structure satisfying the assumptions of Lemma \ref{lemm:transversality_of_kerdalpha} is 
\begin{equation}\label{eq_secondary_type_D}
\beta: M \to \Gr(k, d)/G, \quad x \mapsto f \Big( \big[ p_i(x),p_i(x)^{-1}(\ker d_x \alpha  \cap D_x) \big] \Big).
\end{equation} 
\end{definition}

Note that $\beta$ is independent of the choices of local frames: for a different frame $p_i(x)g$ at $x$, where $g\in G$, we have that 
\begin{align*}
    f([(p_i(x)g),(p_i(x)g)^{-1}(\ker d_x \alpha  \cap D_x)]) & = [g^{-1}p_i(x)^{-1}(\ker d_x \alpha  \cap D_x)]=&\\
    &= \beta(x).
\end{align*}

From the construction of $\beta$, given $\gamma_{xy} \in \cG$, then $\gamma_{xy}$ is also in $\cG'$ --- that is, it satisfies $\gamma_{xy}^{-1}(\ker d_x\alpha \cap D_x) \subset \ker d_y\alpha \cap D_y$ --- if and only if $\beta(x)=\beta(y)$.
Therefore, the following proposition follows.

\begin{proposition}\label{prop_orbit_space_G'}
    The orbits of $\cG'$ coincide with the orbits of $M \times_{(\alpha,\beta)} M$. 
\end{proposition}

\begin{remark}[Stratification by orbit types]
    Having in mind this proposition, to proceed evaluating when $\cG'$ may be a Lie groupoid in a similar fashion to Proposition \ref{prop_definition-G}, it is useful to remember that for any smooth action of a compact Lie group $G$ on a manifold $N$ (in fact, for any proper action), the orbit space admits a canonical stratification. This is usually either called the canonical stratification, or the stratification by orbit types (see \cite{DK2000, Marius_Joao_2018}); the strata consist of the connected components of the projection of orbit types $N_{(H)}$ of $N$ under the quotient $N\to N/G$. The orbit types are defined as $N_{(H)}:=\{x\in M | G_x \text{ is conjugate to } H\}$, for $H$ any subgroup of $G$.
\end{remark}

Sub-Riemannian corank-one 3D and 4D structures are well studied~\cite{contact3d,contact4d}. In those cases, the Lie groupoid $\cG'$ coincides with $\cG$ and is transitive. Let us focus on the 5D corank-one structures, which is the simplest class where $\cG'$ can be different from $\cG$ and non-transitive. A study of sub-Riemannian corank-one 5D structures by different methods can be found in \cite{Ludovic2019}. We have the following proposition in 5D. We will prove in future work an analogous result in arbitrary dimension.

\begin{theorem}
\label{prop:classification_5d}
Let $(M,D,g)$ be a corank-one sub-Riemannian structure with $\dim M = 5$. Assume that the image of its primary and secondary type maps $\alpha: M \to \RR\PP^{n-1}$ and $\beta : M \to \Gr(k,4)/G$ 
are submanifolds of a single stratum of the corresponding stratified space and that $\alpha, \beta$ are submersions onto their images. Then $\cG'\subset \cG$ from \eqref{eq_Lie_groupoid_subriemannian_symmetries_reduced} is a proper regular Lie groupoid. In addition, if $\alpha,\beta$ take values in the open strata, then it is an \'etale groupoid. 
\end{theorem}

\begin{proof}
    Let us study the two type maps. For $\alpha$, as discussed in Example~\ref{ex_charlotte_isotropy}, we have the following possibilities:
    \begin{enumerate}
        \item $\alpha_1(x) = 0$, $\cG_x\cong O(2)\times U(1)$,
        \item $\alpha_1(x) = \alpha_2(x)$, $\cG_x \cong U(2)$,
        \item $\alpha_1(x) \neq \alpha_2(x)$, $\cG_x \cong U(1) \times U(1)$.
    \end{enumerate}
The open stratum corresponds to the third case. In the first two cases above the type map $\alpha$ is constant under the assumption that it take values in a stratum. Therefore, $\ker d_x\alpha = T_xM$ and $\cG'= \cG$ automatically.
Consider the third case. We have the action of $U(1)\cong \SO(2)$ on each eigenspace that corresponds to $\alpha_1(x)$ and $\alpha_2(x)$. Let us call those subspaces $V_1(x)$ and $V_2(x)$. Since the image of $\alpha$ is one-dimensional, from Lemma~\ref{lemm:transversality_of_kerdalpha} and a straightforward dimensional count it follows that $\dim (\ker d_x\alpha \cap D_x) = 3$. Therefore, we have to investigate the orbits of $\Gr(3,4)$ under the $\SO(2)\times \SO(2)$ action. But we know that $\Gr(3,4)\cong \Gr(1,4)\cong \RR\PP^3$ simply by taking orthogonal subspaces.

The classification of the orbits of the $\SO(2)\times \SO(2)$ action on $\mathbb \RR\PP^3$ is rather straightforward. Denote by $\pi_i:D_x \to V_i(x)$ the projections to the corresponding eigenspaces. Since $V_1(x) \oplus V_2(x) = D_x$ we have the following possibilities for the orbit of the line $\ell_x:=(\ker d_x\alpha \cap D_x)^\perp$.
\begin{enumerate}
    \item $\pi_1(\ell_x) = 0$. Then $\ell_x \subset V_2(x)$. Its orbit coincides with the orbit of $\SO(2)$ acting only on $V_2$. Therefore, the stabiliser of $\ell_x$ in this case is given by $\mathbb Z_2 \times \SO(2)$.
    \item $\pi_2(\ell_x) = 0$. This case is exactly as above, we just need to exchange indices $1$ and $2$. The stabiliser group is $\SO(2)\times \mathbb Z_2$.
    \item $\pi_1(\ell_x) \neq  0$ and $\pi_2(\ell_x) \neq  0$. The corresponding orbits are two-dimensional and the isotropy groups are given by $\mathbb Z_2 \times \mathbb Z_2$.
\end{enumerate}
Let us now prove that $\cG'$ is a Lie groupoid. By the assumption that $\beta$ takes values in a single stratum, exactly one of three possibilities for the projections above holds for all $x\in M$. The open stratum corresponds to the third case. The proof follows exactly the lines of the proof of Theorem~\ref{prop_G_regular_Lie_subgroupoid}.
We can construct a frame of $D_x$ adapted to our structure as follows. For the first two cases the adapted frame is given by a direction that generates $\ell_x$, an orthonormal vector inside the same eigenspace, and an orthonormal basis of the other eigenspace.
The choice of this adapted frame is up to an action of $\mathbb Z_2\times \SO(2)$ in case 1 or $\SO(2) \times \mathbb Z_2$ in case 2. For the third possibility it is even simpler: an adapted frame consists of unit length generators of $\pi_1(\ell_x)$ and $\pi_2(\ell_x)$ and two unit length generators of the orthogonal directions inside each eigenspace. The choice of this frame is up to an action of $\mathbb Z_2 \times \mathbb Z_2$.
Using the same argument as in Theorem~\ref{prop_G_regular_Lie_subgroupoid} we can now construct charts of $\cG'$. Indeed, $\pi_i(\ell_x)$ and their orthogonal complements vary smoothly. It allows us to fix a smoothly varying reference frame adapted to our structure in an open neighbourhood. Then we proceed to constructing charts exactly as in the proof of Proposition~\ref{prop_G_regular_Lie_subgroupoid}, which will be of the form
$$
(U_i \times_{(\alpha,\beta)} U_j) \times H_0,
$$
(it is here that $\beta$ being a submersion enters into play) where $H_0$ is either $\mathbb Z_2 \times \SO(2)$, $\SO(2) \times \mathbb Z_2$, or $\mathbb Z_2 \times \mathbb Z_2$ depending on which stratum $\beta$ takes values in. The proof of properness of $\cG'$ follows as that of Proposition \ref{prop_G_proper}, by closedness of $M \times_{(\alpha,\beta)} M\subset M \times M$ and compactness of $H_0$.
\end{proof}

\begin{proposition}\label{prop_representation_for_G'}
    Under the assumptions of Theorem~\ref{prop:classification_5d},
    the representation of $\GL^{\gr}(D)$ on $\gr(TM)$ restricts to a representation of $\cG'$ on $\gr(\Ima \rho_{\cG'})$ if and only if one of the following two conditions is satisfied:
    \begin{enumerate}
        \item the map $\alpha$ or the map $\beta$ are constant maps;
        \item $\pi_i((\ker d_x\alpha \cap D_x)^\perp) \neq  0$ for $i=1,2$ for each $x\in M$ and $\ker d_x\alpha \cap \ker d_x\beta \cap D_x$ is the span $\langle \ell_1, \ell_2\rangle$ of two lines $\ell_i\in V_i$, for $i=1,2$, where $\pi_i$ are projections to the invariant subspaces $V_i$ of $J$.
    \end{enumerate}
\end{proposition}
\begin{proof}
    The charts that we have  constructed for $\cG'$ allow us to identify it locally with the product of a submersion groupoid and a Lie group $(M\times_{(\alpha, \beta)} M) \times G$. Therefore, the representation of $\GL^{\gr}(D)$ on $\gr(TM)$ will restrict to a representation of $\cG'$ on $\gr(\Ima \rho_{\cG'})$ if and only if the representation of the isotropy group bundle $(\cG'_x)_{\{x\in M\}}$ on $\gr(T_xM)$ preserves $\gr(\Ima \rho_{\cG'})$.
  
    If $\alpha$ is a constant map, then $\cG'=\cG$ and it is transitive. The result follows in this case. 
    Consider the other cases. Previously we saw that the orbits of $\cG'$ coincide with the orbits of $M\times_{(\alpha,\beta)} M$ (Proposition~\ref{prop_orbit_space_G'}).
    Since level sets of $(\alpha,\beta)$ are contained in level sets of $\alpha$, we have that the image of the anchor $\rho_\cG'$ of the Lie algebroid of $\cG'$ is $\ker d\beta_{|\ker d\alpha}$.
    Let us assume that $\alpha$ is not a constant map and that $\beta$ is constant. Then we have $\Ima \rho_{\cG'}=\ker d\beta_{|\ker d\alpha} = \ker d\alpha$. The restriction to $\cG'$ of the representation of $\GL^{\gr}(D)$ on $\gr(TM)$ preserves $\gr(\ker d\alpha)$, by definition of $\cG'$, so the result follows also in this case.

    Assume next that $\beta$ is not a constant map. From the proof of Theorem~\ref{prop:classification_5d} we see that the case when $\pi_i((\ker d_x\alpha \cap D_x)^\perp) =0 $ corresponds to a single orbit in the projective space $\RR\PP^3$ and, hence, to a constant $\beta$. It remains to show what happens when $\pi_i((\ker d_x\alpha \cap D_x)^\perp) \neq  0$ for both $i=1,2$, and $\beta$ is not a constant map.
    That corresponds to values of $\beta$ in the open stratum of $\RR\PP^3/\SO(2)\times \SO(2)$.
    
A dimensional count implies that, under the assumption of the Proposition, the image of $(\alpha,\beta)$ is two-dimensional and $\dim\ker d_x(\alpha,\beta) = 3$.
By the same argument as in the proof of Lemma \ref{lemm:transversality_of_kerdalpha}, we show that $D_x + \ker d_x(\alpha,\beta) = \gr(T_xM)$, and so $\dim\ker d_x(\alpha,\beta)\cap D_x = 2$. The isotropy group $\cG'_x$, by construction in the proof of Theorem~\ref{prop:classification_5d}, is the subgroup of $\SO(2)\times \SO(2)$ that preserves $\pi_i((\ker d_x\alpha \cap D_x)^\perp)$ for both $i=1,2$. It is isomorphic to $\mathbb Z_2 \times \mathbb Z_2$, and it acts on $D_x$ by multiplication by +1 or -1 on each eigenspace $V_1, V_2$ of $J_x$.

So the only possibility for $\ker d(\alpha,\beta)$ to be preserved by the isotropy Lie group bundle of $\cG'$ would be if for each $x\in M$ the subspace $\ker d_x(\alpha,\beta) \cap D_x=\ker d_x\alpha \cap \ker d_x\beta \cap D_x$ is either $V_1$, $V_2$, or the product $\ell_1\times \ell_2$ of two lines $\ell_i\in V_i$, for $i=1,2$. But the option of being either $V_1$ or $V_2$ is ruled out by the hypothesis that $\pi_i((\ker d_x\alpha \cap D_x)^\perp) \neq  0$ for $i=1,2$.
\end{proof}

\begin{example}\label{ex_charlotte_G=G'}
    Let us revisit Example~\ref{ex_charlotte_isotropy}. The cases $f'(y_2)\equiv  0$ and $f'(y_2)\equiv  1$ correspond two transitive geometries. So let us consider the complementary cases. When $f'(y_2)>1$, the type map can be essentially written as $\alpha(x_1,x_2,y_1,y_2,z) = f'(y_2)$. Then $(\ker d\alpha \cap D)^\perp = \langle \partial_{y_2}\rangle$ and it is already inside the eigenspace $\langle \partial_{y_2},\partial_{x_2}+ f(y_2)\partial_z\rangle$ for all $(x_1,x_2,y_1,y_2,z)\in \mathbb{R}^5$. The secondary type map $\beta$ in this case is a constant map and $\cG' = \cG$.
\end{example}

\section{Cartan groupoids and sub-Riemannian structures}\label{section_cartan_connections_subriemannian_contact_manifolds}

\subsection{Principal connections and multiplicative Ehresmann connections}\label{subsection_principal_connections}

Recall that a \textbf{principal connection} on a principal $G$-bundle $\pi: P \to M$ is a 1-form $\theta_{\conn} \in \Omega^1 (P,\g)$ such that $\theta_{\conn}$ is $G$-equivariant and $T^\pi P \oplus \ker(\theta_{\conn})= TP$. Equivalently, it is an Ehresmann connection $\cH \subseteq TP$ (i.e.\ a $\pi$-horizontal distribution) which is moreover $G$-invariant.

\begin{example}[principal connections on $G$-structures]\label{ex_principal_connections}
Connections on principal bundles abound. For instance, any linear connection on a vector bundle $E \to M$ can be transformed into a principal connection on its frame bundle $\Fr(E) \to M$. In particular, any $G$-structure $P \subseteq \Fr(M)$ (Example \ref{ex_frame_bundle_general_linear_groupoid}) admits a principal connection.
\end{example}

Since principal bundles are encoded by their gauge groupoids (Proposition \ref{prop_transitive_lie_groupoids_principal_bundles}), principal connections can be transported on connection-like objects on the gauge groupoids as well. More precisely, starting from $\theta_{\conn} \in \Omega^1 (P,\g)$, one obtains a 1-form $\omega_{\conn} \in \Omega^1 (\cG,t^*P[\g])$ on $\cG = \Gauge(P)$. Notice that the associated vector bundle $P[\g] \to M$ is canonically isomorphic (using $\theta_{\conn}$) to the isotropy Lie algebra bundle $\k = \ker(d (s,t) )|_M = \ker(\rho) \to M$. Furthermore, $\omega_{\conn}$ satisfies the following properties:
\begin{itemize}
    \item $\omega_{\conn}$ is pointwise surjective;
    \item $\omega_{\conn}$ is multiplicative (with respect to the natural representation of $\cG$ on the isotropy Lie algebra bundle $\k$ induced by the conjugate action on the isotropy Lie group bundle, i.e.\ obtained by differentiating $C_g: \cG_{s(g)} \to \cG_{t(g)}$ at the identities):
    \[
    (m^* \omega_{\conn})_{(g,h)} = (\pr_1 ^* \omega_{\conn})_{(g,h)} + g \cdot (\pr_2^* \omega_{\conn})_{(g,h)} \text{ for every } (g,h) \in \cG^{(2)};
    \]
 \item $T\cG = \ker(\omega_{\conn}) \oplus K$, with $K := \ker(d (s,t) ) = \ker(ds) \cap \ker(dt) \subseteq T\cG$.
\end{itemize}
More generally, an object of the kind
\begin{equation}\label{eq_multiplicative_Ehresmann_connection}
\omega_{\conn} \in \Omega^1 (\cG,t^*\k),
\end{equation}
with the same properties as above, makes sense on any (not necessarily transitive) Lie groupoid $\cG \tto M$, and it is called a \textbf{multiplicative Ehresmann connection} \cite[Definition 2.4, Proposition 2.8(iii)]{FernandesMarcut23}. To be precise, $\omega_{\conn}$ is should be called a multiplicative Ehresmann connection \textit{with respect to the bundle of ideals $\k$}; the choice of the bundle $\k = \ker(\rho)$ determines also the distribution $K = \ker(ds) \cap \ker(dt) \subseteq T\cG$, obtained by spreading $\k \subseteq A$ using left or right translations. However, in the following we will often omit the reference to the bundle of ideals, since for us it will be always given by the bundle of isotropy Lie algebras. Note also that, assuming the multiplicativity property, the pointwise surjectivity of $\omega_{\conn}$ is equivalent to the condition $\omega_{\conn}|_\k = \id_\k$ (which was the condition originally required in \cite{FernandesMarcut23}).

As for ordinary principal connections, a multiplicative Ehresmann connection $\omega_{\conn}$ can be also equivalently encoded by its kernel, i.e.\ it can be defined as a distribution $E \subseteq T\cG$ satisfying $E \oplus K = T\cG$ and which is multiplicative, i.e.\ $E$ is a Lie subgroupoid of $T\cG \tto TM$ over $TM$.

\begin{example}[Multiplicative Ehresmann connections on general linear groupoids]\label{ex_Ehresmann_connection_GL(TM)}
A principal connection $\theta_{\conn} \in \Omega^1 (P,\gl(n))$ on the principal $\GL(n)$-bundle $\Fr(M) \to M$ corresponds to a multiplicative Ehresmann connection (w.r.t.\ its bundle of ideals $\k$)
\[
\omega_{\conn} \in \Omega^1 (\GL(TM), \Fr(M)[\gl(n)]),
\]
where the associated vector bundle $\Fr(M)[\gl(n)]$ is isomorphic to the bundle of Lie algebras $\k = \ker(\rho)$. 

More generally, any \textit{transitive} Lie groupoid admits a multiplicative Ehresmann connection. This is based on the facts that transitive groupoids are isomorphic to the gauge groupoid of a principal bundle (Proposition \ref{prop_transitive_lie_groupoids_principal_bundles}) and that principal connections on principal bundles always exist (Example \ref{ex_principal_connections}); see \cite[Example 3.4]{FernandesMarcut23} for further details.
\end{example}

\begin{example}\label{ex_MEC_proper_groupoid}
Any proper Lie groupoid admits a multiplicative Ehresmann connection (w.r.t.\ its bundle of ideals $\k$): see \cite[Theorem 4.2]{FernandesMarcut23} for the proof.
\end{example}

\subsection{Cartan connections and Cartan groupoids}\label{section_Cartan_connections}

Let $\widetilde{G}$ be a Lie group and $G \subseteq \widetilde{G}$ be a (closed) Lie subgroup; such data $(G,\tilde{G})$ is also known as a \textbf{Klein geometry}. We denote by $\widetilde{\g}$ and $\g$, respectively, their Lie algebras. Recall that a \textbf{Cartan connection} on a principal $G$-bundle $P \to M$ (w.r.t.\ the Klein geometry $(G,\widetilde{G})$) is a 1-form $\theta \in \Omega^1 (P,\widetilde{\g})$ such that
\begin{itemize}
\item $\theta$ is a pointwise isomorphism, i.e.\ $\theta_p: T_p P \to \widetilde{\g}$ is a linear isomorphism for every $p \in P$;
\item $\theta$ is $G$-equivariant (with respect to the $G$-action on $P$ and the restriction to $G$ of the adjoint action of $\widetilde{G}$ on $\widetilde{\g}$);
\item $\theta(\alpha^\dagger) = \alpha$ for every $\alpha \in \g$, where $\alpha^\dagger \in \mathfrak{X}(P)$ denotes the fundamental vector field associated to $\alpha$ via the infinitesimal $\g$-action on $P$.
\end{itemize}

A \textbf{Cartan geometry} is the data of a Klein geometry, a principal bundle and a Cartan connection; we refer to \cite[Chapter 5]{Sharpe97} and \cite[Chapter 1]{CapSlovak09} for more details. We want however to remark that, unlike ordinary connections on principal bundles (Example \ref{ex_principal_connections}), Cartan connections are much rarer: not only they impose a dimensional constraint (the dimension of $M$ must be equal to $\dim(\tilde{G}) - \dim(G)$), but they induce a vector bundle isomorphism
\begin{equation}\label{eq_Cartan_trivialisation_bundle}
TP \to P \times \widetilde{\g}, \quad \quad v_p \mapsto (p, \theta_p (v_p)),
\end{equation}
i.e.\ the manifold $P$ is forced to be parallelisable.

\begin{example}[Cartan connections on $G$-structures]
Any principal connection on the frame bundle $\Fr(M) \to M$ of an $n$-dimensional manifold $M$ induces a Cartan connection w.r.t.\ $(\GL(n), \RR^n \rtimes \GL(n))$, where $\RR^n \rtimes \GL(n)$ is the semidirect product of $\GL(n)$ and $(\RR^n,+)$, given by the action by matrices on vectors.

Indeed, recalling that $\Fr(M)$ is always endowed with the tautological form $\theta_{\taut} \in \Omega^1 (\Fr(M), \RR^n)$ \eqref{eq_tautological_form_frames}, one can easily check that its sum with any principal connection $\theta_{\conn} \in \Omega^1 (\Fr(M),\gl(n))$, i.e.
\[
\theta := \theta_{\taut} \oplus \theta_{\conn} \in \Omega^1 (\Fr(M), \RR^n \oplus \gl(n)),
\]
is a Cartan connection w.r.t.\ the Klein geometry $(\GL(n), \RR^n \rtimes \GL(n))$.

The analogous construction works starting with any $G$-structure $P \subseteq \Fr(M)$, the restriction $\theta_{\taut} \in \Omega^1 (P,\RR^n)$ of the tautological form of $\Fr(M)$, and a connection form $\theta_{\conn} \in \Omega^1 (P,\g)$. Then the sum $\theta_{\taut} \oplus \theta_{\conn} \in \Omega^1 (P, \RR^n \oplus \g)$ is a Cartan connection w.r.t.\ the Klein geometry $(G, \RR^n \rtimes G)$. This fact is well known in the literature: it appeared first in \cite[Theorem 2]{Kobayashi56}, and is discussed e.g.\ in \cite[Appendix A.2]{Sharpe97}, \cite[Section 1.3]{CapSlovak09} and \cite[Proposition 2.7]{Cattafi21}.
\end{example}

With the same principle outlined for tautological forms and for principal connections, one can transport a Cartan connection on a principal bundle $P \to M$ to an object on its gauge groupoid $\Gauge(P) \tto M$ \cite{Cattafi21}. More precisely, starting from $\theta \in \Omega^1 (P,\widetilde{\g})$, one obtains a 1-form $\omega \in \Omega^1 (\cG, t^*P[\tilde{\g}])$ on $\cG = \Gauge(P)$. Notice that, using the defining properties of a Cartan connection, the vector bundle isomorphism $TP \to P \times \tilde{\g}$ from \eqref{eq_Cartan_trivialisation_bundle} descends to a vector bundle isomorphism
\[
TP/G \cong (P \times \tilde{\g})/G = P[\tilde{\g}]
\]
between the Atiyah algebroid of $P$ (i.e.\ the Lie algebroid of $\cG$) and the coefficient bundle of $\omega$
(this is however not an isomorphism of Lie algebroids, since $P[\tilde{\g}]$ is a bundle of Lie algebras!).
Furthermore, $\omega$ satisfies the following properties:
\begin{itemize}
\item $\omega$ is pointwise surjective;
\item $\omega$ is multiplicative (with respect to the representation of $\cG$ on $P[\tilde{\g}]$ given by the second part of Proposition \ref{prop_transitive_lie_groupoids_principal_bundles}):
\[
(m^* \omega)_{(g,h)} = (\pr_1 ^* \omega)_{(g,h)} + g \cdot (\pr_2^* \omega)_{(g,h)} \text{ for every } (g,h) \in \cG^{(2)};
\]
\item $\ker(\omega) \cap \ker(ds) = \ker(\omega) \cap \ker(dt) = 0$.
\end{itemize}
More generally, an object of the kind
\begin{equation}\label{eq_Cartan_groupoid}
    \omega \in \Omega^1 (\cG, t^*A),
\end{equation}
together with the choice of a representation of $\cG$ on $A$, and satisfying the same properties as above, makes sense on any (not necessarily transitive) Lie groupoid $\cG \tto M$, and the pair $(\cG,\omega)$ is called \textbf{Cartan groupoid}\footnote{We are on purpose avoiding the terminology \textbf{(multiplicative) Cartan connection on a Lie groupoid}, which is used sometimes in the literature with slightly different meanings
.}. Unlike Cartan connections (which do not admit an equivalent point of view involving distributions), a Cartan groupoid can be also equivalently encoded by a distribution $E \subseteq T\cG$ complementary to $\ker(ds)$ and to $\ker(dt)$, 
and which is multiplicative, i.e.\ $E$ is a Lie subgroupoid of $T\cG \tto TM$ over $TM$.

Since in general a Lie groupoid does not admit a canonical representation on its Lie algebroid (its natural "adjoint representation" is actually a representation \textit{up to homotopy}), the existence of Cartan groupoids is severely hindered by the possibility to choose a representation of $\cG$ on $A$. Furthermore, analogously to the isomorphism $TP \cong P \times \widetilde{\g}$ from \eqref{eq_Cartan_trivialisation_bundle} in the case of Cartan geometries, Cartan groupoids force the vector bundle $\ker(ds) \to \cG$ to be trivial, since
\begin{equation}\label{eq_Cartan_trivialisation}
\ker(ds) \to \cG \times_M A, \quad v_g \mapsto (g, \omega_g (v_g) ).
\end{equation}
is a vector bundle isomorphism.

\begin{example}[general linear groupoids as Cartan groupoids]\label{ex_cartan_connection_GL(TM)}

While Cartan groupoids are rare in general, multiplicative Ehresmann connections \eqref{eq_multiplicative_Ehresmann_connection} provide a way to obtain them in the transitive case. In particular, any multiplicative Ehresmann connection on the general linear groupoid $\GL(TM) \tto M$ (which always exists by Example \ref{ex_Ehresmann_connection_GL(TM)}) makes it into a Cartan groupoid.

Indeed, recalling that $\GL(TM)$ is always endowed with the tautological form  $\omega_{\taut} \in \Omega^1 (\GL(TM), t^*TM)$ from \eqref{eq_tautological_form_GL}, one can easily check that its sum with any multiplicative Ehresmann connection $\omega_{\conn} \in \Omega^1 (\GL(TM), t^*\k)$ as in \eqref{eq_multiplicative_Ehresmann_connection}, i.e.
\[
\omega := \omega_{\taut} \oplus \omega_{\conn} \in \Omega^1 (\GL(TM), t^* (TM \oplus \k)),
\]
gives to $\GL(TM)$ the structure of a Cartan groupoid. Indeed, since $\GL(TM)$ is transitive, its Lie algebroid $A = \gl(TM)$ splits globally (non-canonically) as
\[
A \cong \Ima(\rho) \oplus \ker(\rho) = TM \oplus \k, 
\]
therefore $A$ inherits a representation of $\GL(TM)$ as the sum of the representations of $\GL(TM)$ on $TM$ and on $\k$. All the properties of $\omega$ follow automatically from the analogous properties of $\omega_\taut$ and of $\omega_\conn$; in particular,
\[
\ker (ds) \cap \ker(\omega) = \ker (ds) \cap \ker(\omega_\taut) \cap \ker (\omega_\conn) = \ker(ds) \cap \ker (dt) \cap \ker (\omega_\conn)  = 0.
\]

As in Example \ref{ex_Ehresmann_connection_GL(TM)}, we notice that this construction works starting with any \textit{transitive} subgroupoid $\cG \subseteq \GL(TM)$. Indeed, the restriction $\omega_{\taut}|_\cG \in \Omega^1 (\cG, t^*TM)$ of the tautological form of $\GL(TM)$ remains pointwise surjective, as discussed in Remark \ref{rk_restriction_tautological_form_GL_TM}, and its sum with any multiplicative Ehresmann connection $\omega_{\conn} \in \Omega^1 (\GL(TM),t^*\k)$ gives a structure of Cartan groupoid.
\end{example}

\begin{example}[non-transitive subgroupoids of general linear groupoids as Cartan groupoids]\label{ex_nontransitive_subgroupoid_GL_as_Cartan_groupoids}

We focus now on Lie subgroupoids $\cG \subseteq \GL(TM)$ which are not necessarily transitive. If $\cG$ is at least \textit{regular}, the Lie algebroid $A = \Lie(\cG)$ globally splits (non-canonically) as
\[
A \cong \Ima(\rho) \oplus \ker(\rho).
\]
One could therefore restrict not only the tautological form of $\GL(TM)$ to $\cG$, but also its coefficients to $\Ima(\rho)$: the result is a form $\omega_{\taut}|_\cG \in \Omega^1 (\cG, t^*\Ima(\rho))$ satisfying all the properties of $\omega_{\taut}$, including pointwise surjectivity (Remark \ref{rk_restriction_tautological_form_GL_TM}), but excluding multiplicativity: indeed, $\Ima(\rho)$ would not be in general a representation of $\cG$ (see Remark \ref{rk_restriction_tautological_form_GL_TM}). If follows that its sum with any multiplicative Ehresmann connection $\omega_{\conn} \in \Omega^1 (\cG, t^*\ker(\rho))$, i.e.
\[
\omega:= \omega_{\taut}|_\cG \oplus \omega_{\conn} \in \Omega^1 (\cG, t^*A)
\]
makes $\cG$ into a "Cartan groupoid without the multiplicativity property" (since we lack a representation of $\cG$ on $A$).

This point of view is not very satisfactory because the multiplicativity of $\omega$ is arguably the most important condition for a Cartan groupoid (it ties the geometric properties of the form with the algebraic properties of the groupoid). To address this issue, consider the subgroupoid
\begin{equation}\label{eq_G'_subgroupoid}
\cG' := \{ \gamma \in \cG | \gamma: T_x M \to T_y M \text{ sends $\Ima(\rho_x)$ to $\Ima(\rho_y)$ } \} \subseteq \cG \subseteq \GL(TM),
\end{equation}
and notice that it coincides with $\cG$ if and only if the natural representation of $\GL(TM)$ on $TM$ restricts to a subrepresentation of $\cG \subseteq \GL(TM)$ on $\Ima(\rho) \subseteq TM$.

In the case $\cG = \cG'$ we can therefore take the sum of the restriction $\omega_{\taut}|_\cG$ with any multiplicative Ehresmann connection $\omega_{\conn} \in \Omega^1 (\cG, \ker(\rho))$ on $\cG$:
    \[
\omega = \omega_{\taut}|_\cG \oplus \omega_{\conn} \in \Omega^1(\cG, t^*(TM \oplus \ker(\rho)) ).
\]
Since $\cG$ is regular, the image of $\omega$ is
\[
\Ima(\omega) = \Ima(\rho) \oplus \ker(\rho) \cong A,
\]
and it is by construction the sum of two representations of $\cG$. In conclusion, restricting the coefficients of $\omega$ to $A$, we obtain a Cartan groupoid $(\cG, \omega)$.

\

Assume now that $\cG'$ is strictly smaller than $\cG$ but it is smooth. Then the same arguments as above yield a form
\[
\omega_{\taut}|_{\cG'} \in \Omega^1 (\cG', t^* \Ima(\rho)) 
\]    
which is multiplicative with respect to the induced representation of $\cG'$ on $\Ima(\rho)$. However, in general $\Ima(\rho')$ is smaller than $\Ima(\rho)$, hence the sum with any multiplicative Ehresmann connection $\omega_{\conn}' \in \Omega^1 (\cG', \ker(\rho'))$ on $\cG'$
\[
\omega' = \omega_{\taut}|_{\cG'} \oplus \omega_{\conn}' \in \Omega^1(\cG', t^*(\Ima(\rho) \oplus \ker(\rho')) )
\]
may fail to be surjective; indeed, the image of $\omega'$ is $\Ima(\rho') \oplus \ker(\rho') \cong A'$. On the other hand, restricting $\omega$ to its image gives a surjective form which is not multiplicative, since $\Ima(\rho')$ is not a representation of $\cG'$ in general.

In conclusion, unless the orbit foliations of $\cG$ and $\cG'$ coincide (i.e.\ $\Ima(\rho) = \Ima(\rho')$), one needs to further restrict $\cG'$ to the subgroupoid $\cG'' \subseteq \cG'$, whose elements preserve the orbit foliation of $\cG'$, so that $\Ima(\rho')$ becomes by construction a representation of $\cG''$. Also in this case, if $\cG''$ was smooth and $\Ima(\rho'') = \Ima(\rho')$, the process stops, and the sum of the restriction $\omega_{\taut}|_{\cG''} \in \Omega^1 (\cG'', t^* \Ima(\rho'))$ with any multiplicative Erhesmann connection on $\cG''$ yields a Cartan groupoid. If not, one has to further restrict $\cG''$, and continue the process.

What we described so far can be recast in the following diagram; note that the discussion in Example \ref{ex_nontransitive_subgroupoid_GL_as_Cartan_groupoids} could be added as a step number zero, which consists in checking the transitivity of $\cG$, i.e.\ the condition $\Ima(\rho) = TM$.
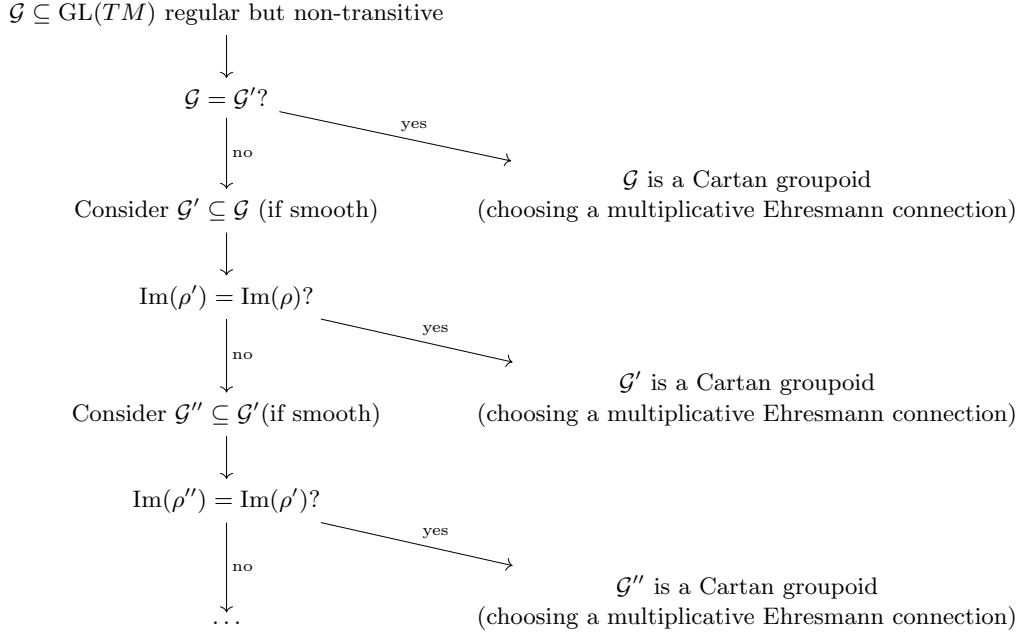
\[\begin{tikzcd}[column sep=0.5em]
	{\cG \subseteq \GL(TM) \text{ regular but non-transitive}} \\
    {\cG = \cG' ?} \\
	{\text{Consider } \cG'\subseteq \cG \text{ (if smooth)}} & {\shortstack{$\cG$ is a Cartan groupoid\\
 (choosing a multiplicative Ehresmann connection)}} \\
	{\Ima(\rho') = \Ima(\rho)?} \\
	{\text{Consider } \cG''\subseteq \cG' \text{(if smooth)}} & {\shortstack{$\cG'$ is a Cartan groupoid\\
 (choosing a multiplicative Ehresmann connection)}} \\
	{\Ima(\rho'') = \Ima(\rho')?} \\
	\ldots & {\shortstack{$\cG''$ is a Cartan groupoid\\
 (choosing a multiplicative Ehresmann connection)}}
	\arrow[from=1-1, to=2-1]
	\arrow["{\text{no}}", from=2-1, to=3-1]
	\arrow["{\text{yes}}", from=2-1, to=3-2]
	\arrow[from=3-1, to=4-1]
	\arrow["{\text{no}}", from=4-1, to=5-1]
	\arrow["{\text{yes}}", from=4-1, to=5-2]
	\arrow[from=5-1, to=6-1]
	\arrow["{\text{no}}", from=6-1, to=7-1]
	\arrow["{\text{yes}}", from=6-1, to=7-2]
\end{tikzcd}\]

    In several situations (e.g.\ in the class of examples we discussed in Example \ref{ex_charlotte_G=G'}) this procedure stops after a finite number of steps. In general this approach, involving repeated reductions leading to a Cartan groupoid, can be viewed as an incarnation of Cartan's equivalence method. A formal discussion of this fact at the infinitesimal level (i.e.\ with the goal of obtaining a Cartan \textit{algebroid}) can be found in \cite{Blaom12}; see in particular the algorithm outlined in Section 2.14, and the case of 3-dimensional sub-Riemannian contact manifolds treated in Section 10.
    \end{example}

\begin{remark}[relations with Pfaffian groupoids]\label{rk_relation_Pfaffian_groupoids}
A Cartan groupoid $(\cG,\omega)$ as in \eqref{eq_Cartan_groupoid} is the same thing as a Pfaffian groupoid with zero symbol. The notion of Pfaffan groupoid was first introduced in \cite{Salazar13} in order to understand the structure behind jet groupoids of Lie pseudogroups, and then further discussed in \cite{Cattafi20} to investigate geometric structures defined by such pseudogroups. A Pfaffian groupoid can be also interpreted as the Lie theoretic version of a Pfaffian fibration, a concept encoding the essential properties of PDEs on jet bundles together with their Cartan forms (see \cite{CattafiCrainicSalazar20}). We give a brief overview here and we refer to the comprehensive monograph \cite{Bookpseudogroups} for further details on all these objects.

A (full) \textbf{Pfaffian groupoid} consists of a Lie groupoid $\cG \tto M$ together with a representation $E \to M$ and a 1-form $\omega \in \Omega^1 (\cG, t^*E)$, such that $\omega$ is multiplicative, pointwise surjective, and the (regular) distribution $\ker(\omega)  \cap \ker(ds)$ is involutive. The fact that $\omega$ is multiplicative and $\Ima(\omega)$ has constant rank implies moreover that $T\cG = \ker(\omega) + \ker(ds) = \ker(\omega) + \ker(dt)$ (see \cite[Proposition 3.4.12]{Cattafi20} or \cite[Lemma A.4]{Cattafi21}).

The \textbf{symbol bundle} of a Pfaffian groupoid $(\cG,\omega)$ is the (involutive) distribution $\mathcal{S}: = (\ker(\omega) \cap \ker(ds))|_M \subseteq A$. This has two extreme cases:
\begin{itemize}
    \item $\mathcal{S} = 0$ (recovering precisely the case of Cartan groupoids \eqref{eq_Cartan_groupoid});
    \item $\mathcal{S} = \ker(\rho)$ (recovering the case of $(\GL(TM),\omega_{\taut})$ \eqref{eq_tautological_form_GL}).
\end{itemize}
Furthermore, one can check that the restriction of $\omega$ to the units induces a vector bundle isomorphism $A/\mathcal{S} \to E$. In particular, in the case of Cartan groupoids, $A$ is isomorphic to $E$, hence the choice of a representation of $\cG$ on its Lie algebroid $A$ becomes part of the data (and it is responsible for restricting the pool of available Cartan groupoids).

Note however that one could formulate a weaker notion of Pfaffian groupoid (without the "fullness" property), requiring that $\omega \in \Omega^1 (\cG,t^*E)$ has constant rank (but not necessarily that it is pointwise surjective) and it satisfies all the other properties above. In this case, one can always replace the coefficient bundle $E \to M$ with the subbundle $F \subseteq E$ given by
\[
F_x := \Ima(\omega_{1_x}) \subseteq E_x, \quad \quad x \in M,
\]
which turns out to be a subrepresentation of $E$; this follows from the multiplicativity of $\omega$ and the fact that $\omega$ has constant rank (see \cite[Lemma 3.4.13]{Cattafi20}). Furthermore, one can also check that $\omega (T\cG) = F$. In conclusion, the same pair $(\cG,\omega)$, together with the subrepresentation $F \subseteq E$, yields a new Pfaffian groupoid which is automatically full w.r.t. $F$ (i.e. the form $\omega \in \Omega^1 (\cG,t^*F)$ is pointwise surjective).
\end{remark}

\subsection{Graded Cartan groupoids associated to sub-Riemannian structures}\label{subsection_graded_Cartan_groupoids}

The notion of Cartan groupoid \eqref{eq_Cartan_groupoid} and the whole discussions in Examples \ref{ex_cartan_connection_GL(TM)} and \ref{ex_nontransitive_subgroupoid_GL_as_Cartan_groupoids} can be upgraded also in the graded setting; we discuss first how (graded) Cartan groupoids can arise in these situations, and then apply these facts to the case of sub-Riemannian structures. 

Through this section we will assume that $D \subseteq TM$ is a bracket-generating equiregular distribution (Definition \ref{def_equiregular_equinilpotent_distributions}); for simplicity we consider it of step 2 (one can therefore think of a corank-one structure, which will be our main setting), but all the definitions and the arguments can be naturally extended to arbitrary steps.

Recall also from equations \eqref{eq_filtration_groupoid}, \eqref{eq_filtration_algebroid} and \eqref{eq_filtration_image_rho} that $D$ defines a filtration $T^i \cG = (ds)^{-1}(D^i) \cap (dt)^{-1}(D^i)$ on the tangent space of any Lie groupoid $\cG \tto M$, a filtration $A^i = T^i \cG \cap A$ on its Lie algebroid $A = \Lie(\cG)$, and a filtration $\Ima(\rho)^i = \Ima(\rho) \cap D^i$ on the tangent space to the orbit foliation. In particular, the graded version of $A$ is
\begin{equation}\label{eq_grading_algebroid}
\gr(A) = A^{-2}/A^{-1} \oplus A^{-1}/A^0 \oplus A^0 = A/\rho^{-1}(D) \oplus \rho^{-1}(D)/\ker(\rho) \oplus \ker(\rho).
\end{equation}
On the other hand, if $A$ is regular, it globally splits (non-canonically) as $\Ima(\rho) \oplus \ker(\rho)$. Moreover, $\rho$ induces the isomorphisms
\[
A/\rho^{-1}(D) \to \Ima(\rho)/(D \cap \Ima(\rho)), \quad [\alpha] \mapsto [\rho(\alpha)]
\]
\[
\rho^{-1}(D)/\ker(\rho) \to D \cap \Ima(\rho), \quad [\alpha] \mapsto \rho(\alpha),
\]
so that
\begin{equation}\label{eq_grA_iso}
\gr(A) \cong \Ima(\rho)/(D \cap \Ima(\rho)) \oplus (D \cap \Ima(\rho)) \oplus \ker(\rho) = \gr(\Ima(\rho)) \oplus \ker(\rho).
\end{equation}

\begin{definition}
A \textbf{graded Cartan groupoid} consists of a Lie groupoid $\cG \tto M$, a representation of $\cG$ on $\gr(A)$, and a collection of smooth vector bundle maps
\[
\omega^{-2}: T\cG \to t^* A/\rho^{-1}(D), \quad \omega^{-1}: T^{-1}\cG \to t^*\rho^{-1}(D)/\ker(\rho), \quad \omega^0: T^{0} \cG \to t^* \ker(\rho),
\]
equivalently described as a differential form
\[
\omega \in \Omega^1 (\cG, t^*\gr(A)),
\]
satisfying the following properties:
\begin{itemize}
\item $\omega$ is pointwise surjective;
    \item $\omega$ is multiplicative (with respect to the representation of $\cG$ on $\gr(A)$), i.e.
    \[
(m^* \omega)_{(g,h)} = (\pr_1 ^* \omega)_{(g,h)} + g \cdot (\pr_2^* \omega)_{(g,h)} \text{ for every } (g,h) \in \cG^{(2)};
\]
\item $\ker(\omega^{-2}) \cap \ker(ds) = \ker(\omega^{-2}) \cap \ker(ds) \cap (dt)^{-1} (D)$ and $\ker(\omega^{-2}) \cap \ker(dt) = \ker(\omega^{-2}) \cap \ker(dt) \cap (ds)^{-1} (D)$;
\item $\ker(\omega^{-1}) \cap \ker(ds) \cap (dt)^{-1} (D) = \ker(\omega^{-1}) \cap \ker(ds) \cap \ker(dt)$ and $\ker(\omega^{-1}) \cap \ker(dt) \cap (ds)^{-1} (D) = \ker(\omega^{-1}) \cap \ker(dt) \cap \ker(ds)$;
\item $\ker(\omega^0) \cap \ker(ds) \cap \ker(dt) = 0$. \qedhere
\end{itemize}
\end{definition}

Analogously to \eqref{eq_Cartan_trivialisation}, any graded Cartan groupoid $(\cG,\omega)$ induces a graded vector bundle isomorphism $\gr(\ker(ds)) \to \cG \times_M \gr(A)$, given by the sum of three isomorphisms
\[
\frac{\ker(ds)}{\ker(ds) \cap (dt)^{-1}(D)} \to \cG \times_M \frac{A}{\rho^{-1}(D)}, \quad [v_g] \mapsto (g, [\omega^{-2}_g (v_g)] ),
\]
\[
\frac{\ker(ds) \cap (dt)^{-1}(D)}{\ker(ds) \cap \ker(dt)} \to \cG \times_M \frac{\rho^{-1}(D)}{\ker(\rho)}, \quad [v_g] \mapsto (g, [\omega^{-1}_g (v_g)]),
\]
\[
\ker(ds) \cap \ker(dt) \to \cG \times_M \ker(\rho), \quad v \mapsto (g, \omega^0_g (v)).
\]

\begin{example}[graded general linear groupoids as graded Cartan groupoids]\label{ex_cartan_connection_GL_gr(TM)}

Consider the Lie groupoid $\GL^{\gr} (D) \tto M$ 
(Proposition \ref{prop_GL_gr_bracket-generating_equiregular}) equipped  with its tautological form $\omega_{\taut} \in \Omega^1 (\GL^\gr (D), t^* \gr(TM))$ from \eqref{eq_tautological_form_GL_gr_TM}. Since $\GL^{\gr} (D)$ is transitive, one can always choose a multiplicative Ehresmann connection $\omega_{\conn} \in \Omega^1 (\GL^{\gr} (D), t^*\ker(\rho))$ (second part of Example \ref{ex_Ehresmann_connection_GL(TM)}); moreover, the sum of $\omega_{\taut}$ and $\omega_{\conn}$ defines a structure 
\[
\omega \in \Omega^1 (\GL^{\gr} (D), t^*\gr(A))
\]
of graded Cartan groupoid on $\GL^{\gr} (D)$. Here $A$ is the Lie algebroid of $\GL^{\gr} (D)$, with $\gr(A)$ from \eqref{eq_grading_algebroid}. Since $A$ is transitive, i.e.\ $\Ima(\rho) = TM$, it follows from \eqref{eq_grA_iso} that
\[
\gr(A) \cong TM/D \oplus D \oplus \ker(\rho) = \gr(TM) \oplus \ker(\rho).
\]
The properties of $\omega$ follow directly from the analogous properties of $\omega_\taut$ and of $\omega_\conn$; in particular, notice that the $0$-component of $\omega$ corresponds precisely to $\omega_\conn|_{T^0 \cG} = \id_\k$.

As in Example \ref{ex_cartan_connection_GL(TM)}, if we consider any \textit{transitive} subgroupoid $\cG \subseteq \GL^{\gr} (D)$, the restriction of the tautological form of $\GL^{\gr} (D)$ (Remark \ref{rk_restriction_tautological_form_GL_gr_TM}), together with any multiplicative Ehresmann connection on $\cG$, defines in the same way a structure of graded Cartan groupoid on $\cG$. 
\end{example}

\begin{example}[non-transitive subgroupoids of graded general linear groupoids as graded Cartan groupoids]\label{ex_nontransitive_subgroupoid_GL_gr_as_Cartan_groupoids}

As in Example \ref{ex_nontransitive_subgroupoid_GL_as_Cartan_groupoids}, we focus now on \textit{regular} Lie subgroupoids $\cG \subseteq \GL^\gr (D)$ which are not necessarily transitive. We begin by noticing that the subgroupoid
\begin{equation}\label{eq_general_G'}
    \cG' := \{ \gamma \in \cG\ |\ \gamma: \gr(T_x M) \to \gr(T_y M) \text{ sends } \gr(\Ima(\rho_x)) \text{ to } \gr(\Ima(\rho_y)) \}\subseteq \cG
\end{equation}
coincides with $\cG$ if and only if the natural representation of $\GL^\gr(TM)$ on $\gr(TM)$ restricts to a subrepresentation of $\cG \subseteq \GL^\gr(TM)$ on $\gr(\Ima(\rho)) \subseteq \gr(TM)$.

If $\cG = \cG'$, consider the sum of the restriction $\omega_{\taut}|_\cG$ with any multiplicative Ehresmann connection $\omega_{\conn} \in \Omega^1 (\cG, \ker(\rho))$ on $\cG$:
    \[
\omega := \omega_{\taut}|_\cG \oplus \omega_{\conn} \in \Omega^1(\cG, t^*(\gr(TM) \oplus \ker(\rho)) ).
\]
The image of $\omega$,
\[
\Ima(\omega) = \gr(\Ima(\rho)) \oplus \ker(\rho),
\]
is isomorphic to $\gr(A)$ via \eqref{eq_grA_iso}, and is the sum of two representations of $\cG$. We can therefore restrict the coefficients of $\omega$ to $\gr(A)$, obtaining a multiplicative and pointwise surjective form. Arguments analogous to those in Example \ref{ex_cartan_connection_GL(TM)} yields a graded Cartan groupoid.

One the other hand, if $\cG'$ is strictly smaller than $\cG$, the subbundle $\gr(\Ima(\rho))$ cannot be a subrepresentation of $\gr(TM)$, hence the image $\Ima(\omega)$ is not a representation of $\cG$. This prevents $(\cG,\omega)$ to be a Cartan groupoid: before restricting its coefficients, $\omega$ cannot be pointwise surjective, but after restricting them to $\Ima(\omega)$, it cannot be multiplicative (since $\Ima(\omega)$ is not a representation).

One needs therefore to assume that $\cG'$ is smooth and to consider the sum of the restriction $\omega_{\taut}|_{\cG'}$ with any multiplicative Ehresmann connection $\omega_{\conn} \in \Omega^1 (\cG', t^* \ker(\rho'))$ on $\cG'$:
    \[
\omega := \omega_{\taut}|_{\cG'} \oplus \omega_{\conn} \in \Omega^1(\cG', t^*(\gr(TM) \oplus \ker(\rho')) ).
\]
If $\gr(\Ima(\rho')) = \gr(\Ima(\rho))$, then the image of $\omega$,
\[
\Ima(\omega) = \gr(\Ima(\rho')) \oplus \ker(\rho'),
\]
is isomorphic to $\gr(A')$ via \eqref{eq_grA_iso}, and is the sum of two representations of $\cG'$. This yields a graded Cartan groupoid $(\cG', \omega \in \Omega^1 (\cG', t^*\gr(A')))$. Otherwise, one needs to apply the previous step to $\cG'$, i.e.\ consider the subgroupoid $\cG'' \subseteq \cG'$, determine if $\gr(\Ima(\rho')) = \gr(\Ima(\rho''))$, and perform a repeated reduction procedure, as in the diagram at the end of Example \ref{ex_nontransitive_subgroupoid_GL_as_Cartan_groupoids}.
\end{example}

\begin{theorem}[Cartan groupoids for 5D sub-Riemannian manifolds]\label{prop_final}
Let $(M,D,g)$ be a corank-one sub-Riemannian manifold of dimension 5. Under the assumptions of Theorem \ref{prop:classification_5d}, the Lie groupoid $\cG'$ from \eqref{eq_Lie_groupoid_subriemannian_symmetries_reduced} admits a structure of Cartan groupoid.
\end{theorem}
\begin{proof}
By Remark \ref{rmk_orbit_foliation_G} the groupoids $\cG'$ from \eqref{eq_Lie_groupoid_subriemannian_symmetries_reduced} and \eqref{eq_G'_subgroupoid} coincide. From the proof of Theorem \ref{prop:classification_5d}, depending on the type maps $\alpha$ and $\beta$, we obtain therefore the following three possible outcomes.
\begin{itemize}
\item $\cG \subseteq \GL^\gr(TM)$ is transitive; then $\cG = \cG'$, which is moreover proper by Proposition \ref{prop_G_proper}, hence it admits a multiplicative Ehresmann connection (Example \ref{ex_MEC_proper_groupoid}). Then $\cG$ is a graded Cartan groupoid by the second part of Example \ref{ex_cartan_connection_GL_gr(TM)}.
\item $\cG$ is étale; then it is trivially a graded Cartan groupoid with the zero form $\omega = 0 \in \Omega^1 (\cG, t^*\gr(A))$, since its Lie algebroid $A$ has rank zero (note that this agrees also with the strategy discussed in Example \ref{ex_nontransitive_subgroupoid_GL_gr_as_Cartan_groupoids}, since multiplicative Ehresmann connections on étale groupoids are automatically zero).
\item $\cG$ is not transitive and not étale; however, in that case $\cG'$ is smooth and proper, hence it admits a multiplicative Ehresmann connection (Example \ref{ex_MEC_proper_groupoid}). Furthermore, the representation of $\cG^\gr(D)$ on $\gr(TM)$ restricts to a representation of $\cG'$ on $\gr(\Ima(\rho'))$ (Proposition \ref{prop_representation_for_G'}), hence $\cG' = \cG''$ is a graded Cartan groupoid by Example \ref{ex_nontransitive_subgroupoid_GL_gr_as_Cartan_groupoids}. \qedhere
\end{itemize}
\end{proof}

\begin{example}
Consider the sub-Riemannian manifold from
Example \ref{ex_charlotte_isotropy}. As discussed in Example \ref{ex_charlotte_G=G'}, depending on the function $f$ the groupoid $\cG$ can be either transitive or equal to $\cG'$; in either case, the theorem above ensures that it has a structure of graded Cartan groupoid.
\end{example}

\addcontentsline{toc}{section}{References}

\printbibliography

\end{document}